\documentclass[preprint,11pt,3p]{elsarticle}
\usepackage{verbatim}
\usepackage{amsmath,amsthm,amsfonts,amssymb,color}
\usepackage{algorithm}
\numberwithin{equation}{section}
\usepackage{mathrsfs}
\usepackage{multirow}
\usepackage{multicol}
\usepackage{algorithmicx}
\def\BState{\State\hskip-\ALG@thistlm}
\makeatother
\usepackage{algpseudocode}

\usepackage{mathtools}
\usepackage{comment}
\usepackage{tabularx,booktabs}

\newcolumntype{L}[1]{>{\raggedright\arraybackslash}p{#1}}
\newcolumntype{C}[1]{>{\centering\arraybackslash}p{#1}}
\newcolumntype{R}[1]{>{\raggedleft\arraybackslash}p{#1}}

\usepackage{cases}

\newcommand{\kronecker}{\raisebox{1pt}{\ensuremath{\:\otimes\:}}}
\usepackage{hyperref}
\usepackage[normalem]{ulem}
\usepackage{xcolor,sectsty}
\usepackage{subfigure}
\newtheorem{theorem}{Theorem}[section]
\newtheorem{lemma}{Lemma}[section]
\newtheorem{definition}{Definition}[section]
\newtheorem{example}{Example}[section]
\newtheorem{remark}{Remark}[section]
\newtheorem{corollary}{Corollary}[section]
\newtheorem{proposition}{Proposition}[section]
\journal{Arxiv}

\newcommand{\D}{{\mathrm D}}
\newcommand{\mc}[1]{\mathcal {#1}}

\newcommand{\n}{{*\!\!_n}}

\newcommand{\p}{{*\!\!_p}}

\newcommand{\s}{{*\!\!_s}}

\newcommand{\m}{{*\!_m}}

\newcommand{\tp}[1]{\textup{#1}}

\newcommand{\ind}[1]{\mathrm{ind}({#1})}  

\newcommand{\rg}{{\mathscr{R}}}
\newcommand{\nl}{{\mathscr{N}}}
\newcommand{\ep}{\scriptsize\mbox{\textcircled{$\dagger$}}}
\newcommand{\core}{\scriptsize\mbox{\textcircled{\#}}}

\begin{document}

\begin{frontmatter}
\title{
{\bf
Computing Tensor Generalized bilateral inverses}}

\author{Ratikanta Behera$^{a,1}$, Jajati Keshari Sahoo$^{b,2}$, Predrag S. Stanimirovi\'c$^{c,d,3}$, Alena Stupina$^{d,4}$,\\ Artem Stupin$^{d,5}$
}
\vspace{.1cm}

\address{
$^{a}$Department of Computational and Data Sciences, Indian Institute of Science, Bangalore, India.\\
\vspace{.1cm}
$^{b}$Department of Mathematics,\\
Birla Institute of Technology $\&$ Science Pilani, K.K. Birla Goa Campus, India\\
\vspace{.1cm}
$^{c}$University of Ni\v s, Faculty of Sciences and Mathematics, Ni\v s, Serbia\\

\vspace{.1cm}
$^{d}$Laboratory "Hybrid Methods of Modelling and Optimization in Complex Systems", Siberian Federal University, Prosp. Svobodny 79, 660041 Krasnoyarsk, Russian Federation; h677hm@gmail.com\\
\vspace{.1cm}
\textit{E-mails}: $^{1}$\texttt{ratikanta@iisc.ac.in}, $^{2}$\texttt{jksahoo@goa.bits-pilani.ac.in}, $^{3}$\texttt{pecko@pmf.ni.ac.rs}, $^{4}$\texttt{h677hm@gmail.com}, $^{5}$\texttt{arstupin@gmail.com}
}

\begin{abstract}
We introduce tensor generalized bilateral inverses (TGBIs) under the Einstein tensor product as an extension of generalized bilateral inverses (GBIs) in the matrix environment. Moreover, the TBGI class includes so far considered composite generalized inverses (CGIs) for matrices and tensors. Applications of TBGIs for solving multilinear systems are presented.
The characterizations and representations of TGBI were studied and verified using a specific algebraic approach. Further, a few characterizations of known CGIs (such as CMP, DMP, MPD, MPCEP, and  CEPMP) are derived. The main properties of the TGBIs ware exploited and verified through numerical examples.
\end{abstract}

\begin{keyword}
Generalized bilateral inverses; Einstein product; Outer inverse; Poission Equations. \\
\vspace{.3cm}
\end{keyword}
\end{frontmatter}

\section{Introduction}\setcounter{equation}{0}
\subsection{Background and motivation}

 As usual, $\mathbb{C}^{I_1\times\cdots\times I_n}$ (resp. $\mathbb{R}^{I_1\times\cdots\times I_n}$) denotes the set containing all orders $n$ complexes (resp. real) tensors.
 The entries of tensor $\mc{M}$ are represented by $(\mc{M})_{i_1...i_n}$.
 Here we use standard tensor notation
\begin{equation*}
 \textbf{M}(m) = M_1 \times\cdots\times M_m,\ \ M_1,\ldots,M_m\in \mathbb N. 
\end{equation*}
  The addition of tensors $\mc{S}, ~\mc{D}\in \mathbb{C}^{I_1\times\cdots\times I_m \times J_1 \times\cdots\times J_n }=\in \mathbb{C}^{I(m),J(n)}$ is defined as
\begin{equation}\label{Eins1}
(\mc{S} + \mc{D})_{i_1...i_mj_1...j_n} =(\mc{S})_{{i_1...i_m}{j_1...j_n}} + (\mc{D})_{{i_1...i_m}{j_1...j_n}}.
\end{equation}
 {Tensor generalized inverses under the {\it Einstein product} has been investigated}  (see \cite{ein,JunWei20,lai,MaLiPre19,WeiMPinverse,WangWei2022}). 
The Einstein product of $\mc{S} \in \mathbb{C}^{\textbf{I}(m) \times \textbf{K}(p)}$ and $\mc{D} \in \mathbb{C}^{\textbf{K}(p) \times \textbf{J}(n)}$, denoted by $\mc{S}\,\p\,\mc{D}$, is defined as
\begin{equation}\label{Eins}
(\mc{S}\,\p\,\mc{D})_{ {i_1...i_mj_1...j_n}}
=\displaystyle\sum_{k_1,\ldots,k_p}(\mc{S})_{{i_1...i_m}{k_1...k_p}}(\mc{D})_{{k_1...k_p}{j_1...j_n}}\in  \mathbb{C}^{\textbf{I}(m) \times \textbf{J}(n)}.
\end{equation}
The range and null space of $\mc{D}\in\mathbb{C}^{\textbf{M}(p)\times \textbf{N}(s)}$ were defined in \cite{WeiDrazin,Stan18} as follows:
\begin{equation*}
\mathscr{R}(\mc{D}) = \left\{\mc{D}\,\s\, \mc{E}:~\mc{E}\in\mathbb{C}^{\textbf{N}(s)}\right\}\mbox{ and } \mathscr{N}(\mc{D})=\left\{\mc{F}\in\mathbb{C}^{ {\textbf{N}(s)}}:~\mc{D}\,\p\,\mc{F}=\mc{O}\in\mathbb{C}^{\textbf{M}(p)}\right\},
\end{equation*}
where $\mc{O}\in\mathbb{C}^{\textbf{M}(p)}$ denotes the zero tensor of the order $\textbf{M}(p)$.
The minimal nonnegative integer $k$ satisfying $\mbox{dim}(\mathscr{R}(\mc{D}^k))= \mbox{dim}(\mathscr{R}(\mc{D}^{k+1}))$ is termed the index \cite{AsishRJ} of $\mc{D} \in \mathbb{C}^{\textbf{I}(n)\times \textbf{I}(n)}$ and denoted by $\mathrm{ind}(\mc{D}).$
In view of the tensor range and  {null} space the authors of \cite{Stan18} introduced the result restated in Lemma \ref{range-stan}.

\begin{lemma} {\em \cite[Lemma 2.2.]{Stan18}}\label{range-stan}
For $\mc{P}\in \mathbb{C}^{\textbf{K}(k)\times \textbf{M}(m)}$, $\mc{Q}\in \mathbb{c}^{\textbf{K}(k)\times \textbf{N}(n)}$
$$\mathscr{R}(\mc{Q})\subseteq\mathscr{R}(\mc{P})\Longleftrightarrow \exists \mc{X}\in \mathbb{C}^{\textbf{M}(m)\times \textbf{N}(n)}:\ \mc{Q}=\mc{P}\,\m\,\mc{X}.$$
\end{lemma}

Tensor generalized inverses (TGIs) have significantly influenced various science fields and, engineering areas, as well as theoretical and applied mathematics \cite{AsishRJ,JSRBPS,sun2018}.
Various TGIs are recently introduced.
Brazell {\it et al}. \cite{BrazellT} defined the tensor inverse under the Einstein product.
Then Sun {\em et al.} in \cite{WeiMPinverse} and, afterwards, Behera and Mishra in \cite{Behera} extended the notion of the tensor inverse to TGIs.
The following tensor equations are used in classifications of TGIs:
Consider a non empty subset   {$\vartheta\subseteq\{1,2,3,4,\ldots\}$. 
A tensor $\mc{Y}$} satisfying the equations $(i)$ for each $i\in \vartheta $ is termed as a $\{\vartheta\}$-inverse of $\mc{D}$.
Such an inverse is denoted by $\mc{D}^{(\vartheta)}$ while $\mc{D}\{\vartheta\}$ is notation for all $\{\vartheta\}$-inverses of $\mc{D}.$
More specifically, generalized inverses for a tensor $\mc{D}$ are presented in Table \ref{Table1T}.

\begin{table}[!htb]
    \centering
   \caption{Recapitulation of generalized inverses}\label{Table1T}
   \vspace*{0.15cm}
		\begin{tabular}{|L{9cm}|L{4.3cm}|C{1.2cm}|}
		\hline
		 {TGIs} & Equations & Notation\\
		\hline
		 Inner inverse of $\mc{D} \in \mathbb C^{\textbf{M}(p)\times \textbf{N}(s)}$ \cite{Behera,WeiMPinverse} & (1)~$\mc{D}\,\s\,\mc{Y}\,\p\,\mc{D} = \mc{D}$ &  $\mc{D}^{(1)}$\\\hline
Outer inverse of $\mc{D} \in \mathbb C^{\textbf{M}(p)\times \textbf{N}(s)}$ \cite{Behera,WeiMPinverse} & (2)~ $\mc{Y}\,\p\,\mc{D}\,\s\,\mc{Y} = \mc{Y}$  & $\mc{D}^{(2)}$ \\\hline
Moore-Penrose (M-P) inverse of $\mc{D} \in \mathbb C^{\textbf{M}(p)\times \textbf{N}(s)}$ \cite{WeiMPinverse}  &   (1)~$\mc{D}\,\s\,\mc{Y}\,\p\,\mc{D} = \mc{D}$ ~~ (2)~ $\mc{Y}\,\p\,\mc{D}\,\s\,\mc{Y} = \mc{Y}$ ~~
(3)~$(\mc{D}\,\s\,\mc{Y})^* = \mc{D}\,\s\,\mc{Y}$ ~~ (4)~$(\mc{Y}\,\p\,\mc{D})^* = \mc{Y}\,\p\,\mc{D}$ & $\mc{D}^\dagger$\\\hline
Drazin inverse of $\mc{D} \in \mathbb C^{\textbf{N}(s)\times \textbf{N}(s)}$, $k=\mathrm{ind}(\mc{D})$ \cite{AsishRJ,WeiDrazin}& $(1^k)$~$\mc{Y}\,\s\,\mc{D}^{k+1} = \mc{D}^{k}$
(2)~$\mc{Y}\,\s\,\mc{D}\,\s\,\mc{Y} = \mc{Y}$
(5)~$\mc{D}\,\s\,\mc{Y} = \mc{Y}\,\s\,\mc{D}$ & $\mc{D}^\D$\\\hline
Core-EP inverse of $\mc{D} \in \mathbb C^{\textbf{N}(s)\times \textbf{N}(s)}$, $k=\mathrm{ind}(\mc{D})$ \cite{JSRBPS} & $(1^k)$~$ {\mc{Y}\,\s\,\mc{D}^{k+1}} = \mc{D}^k$
(6)~$\mc{D}\,\s\,\mc{Y}^2 = \mc{Y}$
(3)~$(\mc{D}\,\s\,\mc{Y})^* = \mc{D}\,\s\,\mc{Y}$ & $\mc{D}^{\ep}$\\\hline
	\end{tabular}
\end{table}
We recall the definition of the outer inverse with known range and null space from \cite{Stan18}.
 A tensor $\mc{Y}$ determined by
$\mc{Y}\,\p\,\mc{D}\,\s\,\mc{Y} = \mc{Y},~ \rg(\mc{Y})=\rg(\mc{B}), \mbox{ and } \nl(\mc{Y})=\nl({\mc{C}})$
is the outer inverse of $\mc{D}\in \mathbb{C}^{\textbf{M}(p) \times \textbf{N}(s)}$ with known range $\rg(\mc{B})$ and null space $\nl(\mc{C})$, in notation $\mc{D}^{(2)}_{\rg(\mc{B}),\nl({\mc{C}})}$.
Further, a tensor $\mc{Y}\in \mc{D}\{\vartheta\}$ satisfying $\rg(\mc{Y})=\rg(\mc{B})$ (resp. $\nl(\mc{Y}) = \nl(\mc{C})$ is denoted by $\mc{D}^{(\vartheta)}_{\rg(\mc{B}),*}$ (resp. $\mc{D}^{(\vartheta)}_{*,\nl(\mc{C})}$).
 {In addition, combining the Drazin inverse from \cite{AsishRJ, WeiDrazin}, the M-P inverse from \cite{WeiMPinverse}, and core inverse from \cite{JSRBPS}, the authors of \cite{ma-ten} defined new types of tensor generalized inverses, i.e., the Drazin M-P (DMP) inverse and the core M-P (CMP) inverse as well as M-P Drazin (MPD) inverse and the M-P core (MPC) inverse.
Perturbation bounds for core and core-EP Inverses of tensor under the Einstein product were considered in \cite{MaFil}.}

Recall the composite tensor generalized inverses of tensor $\mc{D}\in \mathbb C^{\textbf{N}(s)\times \textbf{N}(s)}$.
The hybridization of the M-P with the core-EP inverse was introduced by Chen {\it et al}. \cite{mpcep}, which is known as the MPCEP inverse.
The dual MPCEP is termed as the CEPMP inverse. An extension of the MPCEP and CEPMP inverses to a tensor structure with respect to the Einstein product is provided in Definition \ref{Def1MPCEPT}.
\begin{definition}\label{Def1MPCEPT}
 Let $\mc{D}\in \mathbb C^{\textbf{N}(s)\times \textbf{N}(s)}$. Then:
 \begin{enumerate}
     \item[\rm (i)] the {\rm MPCEP} inverse $\mc{D}^{\dagger,\ep}$ of $\mc{D}$ is defined as $\mc{D}^{\dagger,\ep}=\mc{D}^{\dagger}\,\s\,\mc{D}\,\s\,\mc{D}^{\ep}$;
    \item[\rm (ii)] the {\rm CEPMP} inverse $\mc{D}^{\ep,\dagger}$ of $\mc{D}$ is defined as $\mc{D}^{\ep,\dagger}=\mc{D}^{\ep}\,\s\,\mc{D}\,\s\,\mc{D}^{\dagger}$.
 \end{enumerate}
\end{definition}


On the other hand, one of the most significant contributions of tensor computations is the solution to multilinear systems. Brazell et al. in \cite{BrazellT} studied  {the tensor decompositions}. Motivated by previous investigations and the rapid popularity of composite inverses on matrices and tensors, our intention is to examine properties and representations for tensor generalized bilateral inverses (TGBIs).

{
The generalized bilateral inverse (GBI) of tensors defined by the Einstein products was proposed and investigated in \cite{PublishedSalemi23}.
The dual and self-duality in the GBI are investigated. Guided by the popularity of GBI in matrices and tensor calculus, our motivation was to continue research in the field of GBI on tensors.
}

\smallskip
Main outcomes of current research are as follows.
\begin{enumerate}
\item[$\bullet$] A new class of TGBIs is introduced as an extension of the generalized bilateral inverses (GBIs) in the matrix environment, and a more general form of composite generalized inverses (CGIs)  {are} considered so far in the matrix and tensor case.
\item[$\bullet$] A few characterizations and representations of CGIs, such as CMP, DMP, MPD, MPCEP, and CEPMP inverse, are derived as corollaries.
\item[$\bullet$] An application in solving multilinear systems is studied.
\end{enumerate}

\subsection{Preliminary Results}
\begin{lemma}{\rm \cite[Lemma 4.1 and Theorem 4.2]{JSRBPS}}\label{lm4.2}
Let $\mc{D}\in\mathbb{C}^{\textbf{N}(s) \times \textbf{N}(s)}$ satisfy $\ind{\mc{D}}=k$.
If a tensor $\mc{X}\in\mathbb{C}^{\textbf{N}(s) \times \textbf{N}(s)}$ is defined by
$\mc{X}\,\s\,\mc{D}^{k+1} = \mc{D}^{k}$ and $\mc{D}\,\s\,\mc{X}^{2} = \mc{X}$, then it satisfies:
\begin{enumerate}
\item[\rm (i)]  $\mc{D}\,\s\,\mc{X} = \mc{D}^{p}\,\s\,\mc{X}^{p},\ p\in\mathbb{N}$.
\item[\rm (ii)] $\mc{X}\,\s\,\mc{D}\,\s\,\mc{X} = \mc{X}$.
\item [\rm (iii)] $\mc{D}^{\ep} = \mc{D}^\D\,\s\,\mc{D}^{l}\,\s\,\left(\mc{D}^l\right)^{\dagger},$ \ $l\geq k$.
\end{enumerate}
\end{lemma}
In \cite[Proposition 2.1 (b)]{ma-ten} it has been proved that the DMP and MPD inverses belong to the class of outer inverses.

\subsection{Outline}
The definitions, main properties, and representations of TBGIs are discussed in Section \ref{SecTGBI}. The properties and representations of the TGBIs are discussed in Subsection \ref{SubSecPropTGBI}. The properties of CMP, DMP, MPD, MPCEP, and CEPMP inverses are discussed in Subsection \ref{SubSecCGI}. The applications of TGBIs for solving multilinear systems  {are} presented in Section \ref{SecSMS}. Numerical examples  {are} presented in Section \ref{SecExamples}. The final observations are presented in  Section \ref{SecConclusion}.

\section{Tensor generalized bilateral inverses}\label{SecTGBI}

In this section, we extend the notion of the bilateral inverse \cite{bilat} from the matrix case  {to} tensors under the framework of the Einstein product.
The tensor version of the bilateral inverse is more general in nature, and the composite generalized inverses can be obtained as a particular case.
 {
\begin{definition} [\cite{PublishedSalemi23}, Definition 24]  \label{dfnBilatT}
    Let $\mc{D}\in \mathbb C^{\textbf{M}(p)\times \textbf{N}(s)}$, and $\mc{X},\mc{Y}\in \mc{D}\{1\}\cup \mc{D}\{2\}$.
The class of TGBIs of $\mc{D}$ is defined as     $\mc{X}\,\p\,\mc{D}\,\s\,\mc{Y}$.
\end{definition}
}
\begin{remark}
If $\mc{X}\,\p\,\mc{D}\,\s\,\mc{Y}$ is a generalized  bilateral inverse of the tensor $\mc{D}\in \mathbb C^{\textbf{M}(p)\times \textbf{N}(s)}$, then $\mc{Y}\,\p\,\mc{D}\,\s\,\mc{X}$ is the dual of $\mc{X}\,\p\,\mc{D}\,\s\,\mc{Y}$.
\end{remark}

\subsection{Properties of tensor generalized bilateral inverses}\label{SubSecPropTGBI}

The following results are verified using Definition \ref{dfnBilatT}.
\begin{proposition}\label{prop2.2}
  Let $\mc{D}\in \mathbb C^{\textbf{M}(p)\times \textbf{N}(s)}$. Then the following holds:
  \begin{enumerate}
      \item[\rm (i)] $\mc{X},\mc{Y}\in \mc{D}\{1,2\}\Longrightarrow \mc{X}\p\mc{D}\,\s\,\mc{Y}\in \mc{D}\{1,2\}\wedge \mc{Y}\,\p\,\mc{D}\,\s\,\mc{X}\in \mc{D}\{1,2\}$.
      \item[\rm (ii)] $\mc{Y}\in \mc{D}\{1\}\wedge \mc{X}\in \mc{D}\{2\}\Longrightarrow \mc{X}\,\p\,\mc{D}\,\s\,\mc{Y}\in \mc{D}\{2\}\wedge \mc{Y}\,\p\,\mc{D}\,\s\,\mc{X}\in \mc{D}\{2\}$.
      \item[\rm (iii)] $\mc{X},\mc{Y}\in \mc{D}\{1\}\Longrightarrow \mc{X}\,\p\,\mc{D}\,\s\,\mc{Y}\in \mc{D}\{1\}\wedge \mc{Y}\,\p\,\mc{D}\,\s\,\mc{X}\in \mc{D}\{1\}$.
  \end{enumerate}
\end{proposition}

Some further properties of TGBIs are investigated in \cite{PublishedSalemi23}
 {
\begin{theorem} {\rm \cite[Corollary  26]{{PublishedSalemi23}}}\label{thm3.2}
  Let $\mc{D}\in \mathbb C^{\textbf{M}(p)\times \textbf{N}(s)}$, $\mc{X}\in \mc{D}\{2\}$ and $\mc{Y}\in \mc{D}\{1\}$.
  Then $\mc{Y}\,\p\,\mc{D}\,\s\,\mc{X}$ is the unique solver of the system tensor equations
  \begin{equation}\label{eq3.1}
\mc{Z}\,\p\,\mc{D}\,\s\,\mc{Z}=\mc{Z}, \ \ \mc{D}\,\s\,\mc{Z}=\mc{D}\,\s\,\mc{X}, \ \  {\mc{Z}\,\p\,\mc{D}} = \mc{Y}\,\p\,\mc{D}\,\s\,\mc{X}\,\p\,\mc{D}.
  \end{equation}
\end{theorem}
\begin{theorem} {\rm \cite[Theorem  25]{{PublishedSalemi23}}}\label{thm3.4}
  Let $\mc{D}\in \mathbb C^{\textbf{M}(p)\times \textbf{N}(s)}$, $\mc{X}\in \mc{D}\{2\}$, and $\mc{Y}\in \mc{D}\{1\}$.
  Then  {$\mc{X}\,\p\,\mc{D}\,\s\,\mc{Y}$} is the unique solution of the following tensor equations:
  \begin{equation}\label{eq2.1}
  \mc{Z}\,\p\,\mc{D}\,\s\,\mc{Z}=\mc{Z},\ \ \mc{Z}\,\p\,\mc{D}=\mc{X}\,\p\,\mc{D}, \ \ \mc{D}\,\n\,\mc{Z} = \mc{D}\,\s\,\mc{X}\,\p\,\mc{D}\,\s\,\mc{Y}.
  \end{equation}
\end{theorem}
}
The composite generalized inverse of tensors can be seen as a particular case of the generalized bilateral inverses.
A summarization of main composite outer inverses is arranged in Table \ref{tab:cmpo}.
\begin{table}[htb!]
\caption{Composite generalized inverses for the tensor $\mc{D}\in \mathbb C^{\textbf{N}(s)\times \textbf{N}(s)}$}
    \label{tab:cmpo}
    \centering
    \begin{tabular}{|l|c|c |c|}
   \hline
 Title     & Definitions &  Choices for $\mc{X}$ and $\mc{Y}$ & Restriction\\
     \hline
  DMP inverse \cite{ma-ten}  & $\mc{D}^{\D,\dagger}=\mc{D}^\D\,\s\,\mc{D}\,\s\,\mc{D}^{\dagger}$   & $\mc{X}=\mc{D}^\D$,\ $\mc{Y}=\mc{D}^{\dagger}$ & $\mathrm{ind}(\mc{D}) =k$\\
   \hline
  MPD inverse  \cite{ma-ten} &$\mc{D}^{\dagger,\D}=\mc{D}^{\dagger}\,\s\,\mc{D}\,\s\,\mc{D}^{\D}$   & $\mc{X}=\mc{D}^{\dagger}$,\ $\mc{Y}=\mc{D}^{\D}$ & $\mathrm{ind}(\mc{D}) =k$\\
   \hline
    CMP inverse \cite{ma-ten} & $\mc{D}^{c,\dagger}=\mc{D}^{\dagger}\,\s\,\mc{D}\,\s\,\mc{D}^{\D,\dagger}=\mc{D}^{\dagger,\D}\,\s\,\mc{D}\,\s\,\mc{D}^{\dagger}$  & $\mc{X}=\mc{D}^{\dagger}$, \ $\mc{Y}=\mc{D}^{\D,\dagger}$& $\mathrm{ind}(\mc{D}) =k$ \\\hline
   MPCEP inverse &  $\mc{D}^{\dagger,\ep}=\mc{D}^{\dagger}\,\s\,\mc{D}\,\s\,\mc{D}^{\ep}$   & $\mc{X}=\mc{D}^{\dagger}$,\ $\mc{Y}=\mc{D}^{\ep}$ & - \\
   \hline
    CEPMP inverse &  $\mc{D}^{\ep,\dagger}=\mc{D}^{\ep}\,\s\,\mc{D}\,\s\,\mc{D}^{\dagger}$   & $\mc{X}=\mc{D}^{\ep}$,\ $\mc{Y}=\mc{D}^{\dagger}$ & - \\
   \hline
\end{tabular}
\end{table}

\begin{example}
 Let
$~\mc{D}=(a_{ijkl})
 \in \mathbb{R}^{\overline{2\times3}\times\overline{2\times3}}$ with entries
\begin{eqnarray*}
a_{ij11} =
    \begin{pmatrix}
    1 & -1 & 1 \\
    1 & 1 &  -1
    \end{pmatrix},~
a_{ij12} =
    \begin{pmatrix}
     1 & 0 & -1\\
     1 & 0 & 0
    \end{pmatrix},~
a_{ij13} =
    \begin{pmatrix}
     1 & 0 & -1\\
     1 & -1 & 0
\end{pmatrix},\\
a_{ij21} =
    \begin{pmatrix}
     0 & 0 & 1\\
     -1 & 0 & -1
    \end{pmatrix},~
    a_{ij22} =
    \begin{pmatrix}
    0 & -1 & -1 \\
    1 & -1 &  0
    \end{pmatrix},~
a_{ij23} =
    \begin{pmatrix}
     1 & 1 & -1\\
     1 & 0 & 1
    \end{pmatrix}.
\end{eqnarray*}
The entries of Moore-Penrose inverse of $\mc{D}$ are as follows:
\begin{eqnarray*}
(\mc{D}^{\dagger})_{ij11} =
    \begin{pmatrix}
    1/8 & 1/8 & 5/16 \\
    5/16 & -3/16 &  3/16
    \end{pmatrix},~
(\mc{D}^{\dagger})_{ij12} =
    \begin{pmatrix}
     -1/8 & -1/8 & 3/16\\
     3/16 & -5/16 & 5/16
    \end{pmatrix},\\
(\mc{D}^{\dagger})_{ij13} =
    \begin{pmatrix}
     1/2 & -3/2 & 1/2\\
     0 & 0 & 1/2
\end{pmatrix},~
(\mc{D}^{\dagger})_{ij21} =
    \begin{pmatrix}
     3/8 & -5/8 & 3/16\\
     -5/16 & 3/16 & 5/16
    \end{pmatrix},\\
    (\mc{D}^{\dagger})_{ij22} =
    \begin{pmatrix}
    0 & 1 & -3/4 \\
    -1/4 & -1/4 &  -1/4
    \end{pmatrix},~
(\mc{D}^{\dagger})_{ij23} =
    \begin{pmatrix}
     1/8 & -7/8 & 1/16\\
     -7/16 & 1/16 & 7/16
    \end{pmatrix}.
\end{eqnarray*}
Clearly, $\mathrm{ind}(\mc{D})=3$. The entries of the Drazin inverse of  $\mc{D}$ are computed as below.
\begin{eqnarray*}
(\mc{D}^\D)_{ij11} =
    \begin{pmatrix}
    1/16 & 37/64 & -123/256 \\
    19/256 & 23/256 &  151/256
    \end{pmatrix},~
(\mc{D}^\D)_{ij12} =
    \begin{pmatrix}
     -1/16 & 19/64 & -93/256\\
     5/256 & -31/256 & 97/256
    \end{pmatrix},
\end{eqnarray*}
 \begin{eqnarray*}
(\mc{D}^\D)_{ij13} =
    \begin{pmatrix}
     -1/8 & 5/32 & -39/128\\
     -1/128 & -29/128 & 35/128
\end{pmatrix},
   (\mc{D}^\D)_{ij21} =
    \begin{pmatrix}
     1/16 & 9/64 & -15/256\\
     7/256 & 27/256 & 27/256
    \end{pmatrix},
\end{eqnarray*}
 \begin{eqnarray*}
    (\mc{D}^\D)_{ij22} =
    \begin{pmatrix}
    1/4 & -1/16 & -13/64 \\
    21/64 & -1/64 &  -1/64
    \end{pmatrix},~
(\mc{D}^\D)_{ij23} =
    \begin{pmatrix}
     -7/16 & 5/64 & -11/256\\
     -93/256 & -89/256 & 39/256
    \end{pmatrix}.
\end{eqnarray*}

Similarly, the entries of the core-EP inverse are as follows:
\begin{eqnarray*}
(\mc{D}^{\ep})_{ij11} =
    \begin{pmatrix}
    \frac{163}{1011} & \frac{733}{2816} & \frac{-445}{1726} \\
    \frac{1021}{6291} & \frac{779}{6405} &  \frac{845}{3233}
    \end{pmatrix},~
(\mc{D}^{\ep})_{ij12} =
    \begin{pmatrix}
     -349/2688 & 493/4236 & -2002/9523\\
     -412/9359 & -191/1704 & 92/455
    \end{pmatrix},
    \end{eqnarray*}
 \begin{eqnarray*}
(\mc{D}^{\ep})_{ij13} =
    \begin{pmatrix}
     367/3039 & -120/1921 & -284/2425\\
     458/2543 & -277/9589 & -30/9589
\end{pmatrix},~
   (\mc{D}^{\ep})_{ij21} =
    \begin{pmatrix}
     3/8 & -5/8 & 3/16\\
     -5/16 & 3/16 & 5/16
    \end{pmatrix},
    \end{eqnarray*}
 \begin{eqnarray*}
    (\mc{D}^{\ep})_{ij22} =
    \begin{pmatrix}
    0 & 1 & -3/4 \\
    -1/4 & -1/4 &  -1/4
    \end{pmatrix},~
(\mc{D}^{\ep})_{ij23} =
    \begin{pmatrix}
     1/8 & -7/8 & 1/16\\
     -7/16 & 1/16 & 7/16
    \end{pmatrix}.
\end{eqnarray*}
Next, we compute the generalized bilateral inverse by restricting the inner and outer inverses.
\begin{eqnarray*}
(\mc{D}^{\dagger,\D})_{ij11} =
   \begin{pmatrix}
       1/16     &     37/64     &    -11/64   \\
     -15/64      &    -7/32      &     9/32
    \end{pmatrix},~
(\mc{D}^{\dagger,\D})_{ij21} =
  \begin{pmatrix}
       1/16    &       9/64      &     1/64    \\
      -3/64     &      1/32      &     1/32
    \end{pmatrix},
\end{eqnarray*}
\begin{eqnarray*}
(\mc{D}^{\dagger,\D})_{ij12} =
  \begin{pmatrix}
      -1/16     &     19/64     &    -13/64   \\
      -9/64     &     -9/32      &     7/32
    \end{pmatrix},~
(\mc{D}^{\dagger,\D})_{ij22} =
  \begin{pmatrix}
       1/4      &     -1/16    &      -1/16    \\
       3/16     &     -1/8      &     -1/8
    \end{pmatrix},
\end{eqnarray*}
\begin{eqnarray*}
(\mc{D}^{\dagger,\D})_{ij13} =
  \begin{pmatrix}
      -1/8      &      5/32     &     -7/32  \\
      -3/32      &    -5/16      &     3/16
    \end{pmatrix},~
(\mc{D}^{\dagger,\D})_{ij23}
  \begin{pmatrix}
      -7/16     &      5/64     &    -11/64 \\
     -15/64      &    -7/32     &      9/32
    \end{pmatrix},
\end{eqnarray*}

\begin{eqnarray*}
(\mc{D}^{\D,\dagger})_{ij11}=
  \begin{pmatrix}
       5/32     &     67/128    &   -233/512 \\
      81/512    &     77/512    &    269/512
    \end{pmatrix},~
(\mc{D}^{\D,\dagger})_{ij21} =
  \begin{pmatrix}
      -1/32     &     25/128    &    -43/512  \\
     -29/512    &     23/512    &     87/512
    \end{pmatrix},
\end{eqnarray*}
\begin{eqnarray*}
(\mc{D}^{\D,\dagger})_{ij12} =
  \begin{pmatrix}
      -5/32     &     45/128  &     -199/512  \\
     -33/512    &    -93/512     &   227/512
    \end{pmatrix},~
(\mc{D}^{\D,\dagger})_{ij22} =
  \begin{pmatrix}
       1/4       &    -1/16    &     -13/64  \\
      21/64      &     1/64    &       1/64
    \end{pmatrix},
\end{eqnarray*}
\begin{eqnarray*}
(\mc{D}^{\D,\dagger})_{ij13}=
  \begin{pmatrix}
      -1/8       &     5/32     &    -39/128  \\
      -1/128     &   -29/128    &     35/128
    \end{pmatrix},~
(\mc{D}^{\D,\dagger})_{ij23} =
  \begin{pmatrix}
     -11/32      &     3/128   &      -9/512 \\
    -143/512    &   -147/512    &     45/512
    \end{pmatrix},
\end{eqnarray*}
\begin{eqnarray*}
(\mc{D}^{\dagger,\ep})_{ij11} =
  \begin{pmatrix}
     \frac{163}{1011}   &      \frac{733}{2816}   &    \frac{-232}{4067}  \\
   \frac{-369}{9589}  &      \frac{-172}{2173}  &       \frac{232}{3829}
    \end{pmatrix},~
(\mc{D}^{\dagger,\ep})_{ij21}  =
  \begin{pmatrix}
     411/3445   &    564/2833   &     65/1589  \\
     -55/1006   &    447/6343   &    229/9121
    \end{pmatrix},
\end{eqnarray*}

\begin{eqnarray*}
(\mc{D}^{\dagger,\ep})_{ij12} =
  \begin{pmatrix}
     \frac{ -349}{2688}   &      \frac{493}{4236}    &    \frac{-427}{2589}  \\
      \frac{-225}{2519}   &    \frac{-1452}{6245}   &      \frac{257}{1638}
    \end{pmatrix},~
(\mc{D}^{\dagger,\ep})_{ij22}=
  \begin{pmatrix}
     299/2177   &    589/2330   &   -647/3241  \\
    -118/13841  &   -377/1265    &   581/5434
    \end{pmatrix},
\end{eqnarray*}
\begin{eqnarray*}
(\mc{D}^{\dagger,\ep})_{ij13} =
  \begin{pmatrix}
     \frac{367}{3039}   &    \frac{-120}{1921}  &     \frac{-253}{4979}  \\
      \frac{343}{3014}   &     \frac{-89}{935}   &     \frac{-181}{2607}
    \end{pmatrix},~
(\mc{D}^{\dagger,\ep})_{ij23}=
  \begin{pmatrix}
    -371/2160   &    150/2719  &    -244/3643  \\
    -559/5298    &  -107/1291   &    453/3731
    \end{pmatrix},
\end{eqnarray*}
\begin{eqnarray*}
(A^{\ep, \dagger})_{ij11} =
  \begin{pmatrix}
      \frac{163}{1011}   &     \frac{733}{2816}  &     \frac{-445}{1726}  \\
     \frac{1021}{6291}   &     \frac{779}{6405}   &     \frac{845}{3233}
    \end{pmatrix},~
(A^{\ep, \dagger})_{ij21} =
  \begin{pmatrix}
     411/3445   &    564/2833  &    -289/4009  \\
     301/5161    &   162/883    &   1037/7509
    \end{pmatrix},
\end{eqnarray*}
\begin{eqnarray*}
(A^{\ep, \dagger})_{ij12} =
  \begin{pmatrix}
    \frac{ -349}{2688}   &     \frac{623}{5353}  &    \frac{-2002}{9523}  \\
     \frac{-412}{9359}  &     \frac{-319}{1704}   &      \frac{92}{455}
    \end{pmatrix},~
    (A^{\ep, \dagger})_{ij22} =
  \begin{pmatrix}
     299/2177   &    589/2330  &    -467/1038  \\
     271/1121   &   -191/4000   &    509/1425
    \end{pmatrix},
\end{eqnarray*}
\begin{eqnarray*}
(A^{\ep, \dagger})_{ij13} =
  \begin{pmatrix}
      \frac{367}{3039}   &    \frac{-120}{1921}  &     \frac{-284}{2425}  \\
      \frac{458}{2543}   &    \frac{-277}{9589}   &     \frac{-30}{9589}
    \end{pmatrix},~
(A^{\ep, \dagger})_{ij23} =
  \begin{pmatrix}
    -371/2160   &    150/2719   &   -385/15718 \\
    -358/2419   &   -644/5137   &    263/3332
    \end{pmatrix},
\end{eqnarray*}
\begin{eqnarray*}
(A^{c,\dagger})_{ij11} =
  \begin{pmatrix}
       5/32     &     67/128    &    -17/128   \\
     -21/128    &    -11/64     &     13/64
    \end{pmatrix},~
(A^{c,\dagger})_{ij21} =
  \begin{pmatrix}
      -1/32      &    25/128    &     -3/128  \\
     -15/128     &    -1/64    &       7/64
    \end{pmatrix},
\end{eqnarray*}
\begin{eqnarray*}
(A^{c,\dagger})_{ij12} =
  \begin{pmatrix}
      -5/32     &     45/128    &    -31/128  \\
     -27/128    &    -21/64     &     19/64
    \end{pmatrix},~
(A^{c,\dagger})_{ij22} =
  \begin{pmatrix}
       1/4     &      -1/16     &     -1/16    \\
       3/16     &     -1/8      &     -1/8
    \end{pmatrix},
\end{eqnarray*}
\begin{eqnarray*}
(A^{c,\dagger})_{ij13} =
  \begin{pmatrix}
      -1/8     &       5/32    &      -7/32 \\
      -3/32     &     -5/16     &      3/16
    \end{pmatrix},~
(A^{c,\dagger})_{ij23} =
  \begin{pmatrix}
     -11/32    &       3/128   &     -17/128 \\
     -21/128    &    -11/64    &      13/64
    \end{pmatrix}.
    \end{eqnarray*}
 \end{example}

In the remaining results of this subsection, we discuss the equality of TGBIs and respective duals. The assumption $\mc{D}\in \mathbb C^{\textbf{N}(s)\times \textbf{N}(s)}$ will be default.

 {
\begin{theorem}  [\cite{PublishedSalemi23}, Theorem 28] \label{thm3.5}
  Let   $\mc{X}\in \mc{D}\{2\}$ and $\mc{Y}\in \mc{D}\{1\}$.  The following statements are equivalent:
  \begin{enumerate}
      \item[\rm (i)] $\mc{X}\,\s\,\mc{D}\,\s\,\mc{Y}=\mc{Y}\,\s\,\mc{D}\,\s\,\mc{X}$.
      \item[\rm (ii)] $\mc{X}=\mc{X}\,\s\,\mc{D}\,\s\,\mc{Y}=\mc{Y}\,\s\,\mc{D}\,\s\,\mc{X}$.
      \item[\rm (iii)] $\mathscr{N}(\mc{D}\,\s\,\mc{Y})\subseteq \nl(\mc{X})\wedge \rg(\mc{X})\subseteq\rg(\mc{Y}\,\s\,\mc{D})$.
  \end{enumerate}
\end{theorem}
}
{\color{red}
}

\begin{corollary}
   The following assertions are mutually equivalent: 
    \begin{enumerate}
      \item[\rm (i)] $\mc{D}^{\ep}\,\s\,\mc{D}\,\s\,\mc{D}^{\dagger}=\mc{D}^{\dagger}\,\s\,\mc{D}\,\s\,\mc{D}^{\ep}$.
      \item[\rm (ii)] $\mc{D}^{\ep}=\mc{D}^{\ep}\,\s\,\mc{D}\,\s\,\mc{D}^{\dagger}=\mc{D}^{\dagger}\,\s\,\mc{D}\,\s\,\mc{D}^{\ep}$.
    \end{enumerate}
\end{corollary}

\begin{proposition}
The following equivalence holds for 
$\mc{X},\mc{Z}\in \mc{D}\{1\}$.
  \begin{equation*}
 \mc{X}\,\s\,\mc{D}\,\s\,\mc{Z}= \mc{Z}\,\s\,\mc{D}\,\s\,\mc{X} \Longleftrightarrow \mc{X}\,\s\,\mc{D}=\mc{Z}\,\s\,\mc{D} \wedge  \mc{D}\,\s\,\mc{Z}=\mc{D}\,\s\,\mc{X}.
  \end{equation*}
\end{proposition}
\begin{proof}
 Let $\mc{X}\,\s\,\mc{D}\,\s\,\mc{Z}= \mc{Z}\,\s\,\mc{D}\,\s\,\mc{X} $.
 Then
 \begin{equation*}     \mc{X}\,\s\,\mc{D}=\mc{X}\,\s\,\mc{D}\,\s\, \mc{Z}\,\s\,\mc{D}=\mc{Z}\,\s\,\mc{D}\,\s\,\mc{X}\,\s\,\mc{D}=\mc{Z}\,\s\,\mc{D},
 \end{equation*}
 and
 \begin{equation*}
  \mc{D}\,\s\,\mc{Z}=\mc{D}\,\s\,\mc{X}\,\s\,\mc{D}\,\s\,\mc{Z}= \mc{D}\,\s\,\mc{Z}\,\s\,\mc{D}\,\s\,\mc{X}=\mc{D}\,\s\,\mc{X}.
 \end{equation*}
 Conversely, let $\mc{X}\,\s\,\mc{D}=\mc{Z}\,\s\,\mc{D} \mbox{ and }   \mc{D}\,\s\,\mc{Z}=\mc{D}\,\s\,\mc{X}$.
 Then
 {$$\mc{X}\,\s\,\mc{D}\,\s\,\mc{Z}= \mc{Z}\,\s\,\mc{D}\,\s\,\mc{Z} =   \mc{Z}\,\s\,\mc{D}\,\s\,\mc{X},$$}
which confirms the statement.
\end{proof}
\begin{theorem}
 The subsequent equivalencies hold for $\mathrm{ind}(\mc{D})=k$:
  \begin{enumerate}
      \item[\rm (i)] $\mc{D}^{c,\dagger}=\mc{D}^{\dagger,\D}\Longleftrightarrow \nl(\mc{D}^{\dagger})\subseteq \nl(\mc{D}^{k})$.
      \item[\rm (ii)]  $\mc{D}^{c,\dagger}=\mc{D}^{\D,\dagger}\Longleftrightarrow \rg(\mc{D}^k)\subseteq \rg(\mc{D}^{\dagger})$.
  \end{enumerate}
\end{theorem}
\begin{proof}
    (i) Let $\mc{D}^{\dagger}\,\s\,\mc{D}\,\s\,\mc{D}^\D=\mc{D}^{\dagger}\,\s\,\mc{D}\,\s\,\mc{D}^\D\,\s\,\mc{D}\,\s\,\mc{D}^{\dagger}$.
    Pre-multiplying the last equality by $\mc{D}$, we get $\mc{D}\,\s\,\mc{D}^\D=\mc{D}\,\s\,\mc{D}^\D\,\s\,\mc{D}\,\s\,\mc{D}^{\dagger}$.
     If  $\mc{Q}\in \nl(\mc{D}^{\dagger})$ then $\mc{D}^{\dagger}\,\s\,\mc{Q}=0$.
Now $\nl(\mc{D}^{\dagger})\subseteq \nl(\mc{D}^{k})$ is valid due to
    \begin{eqnarray*}
        \mc{D}^k=\mc{D}^k\,\s\,\mc{D}\,\s\,\mc{D}^\D=\mc{D}^k\,\s\,\mc{D}\,\s\,\mc{D}^\D\,\s\,\mc{D}\,\s\,\mc{D}^{\dagger}.
    \end{eqnarray*}
Conversely, let $\mc{X}=D^{\dagger,\D}$ and $\mc{Y}=\mc{D}^{\dagger}$.
According to the Theorem \ref{thm3.5}, it is sufficient to show $\nl(\mc{D}\,\s\,\mc{Y})\subseteq \nl(\mc{X})$ and $\rg(\mc{X})\subseteq\rg(\mc{Y}\,\s\,\mc{D})$.
Clearly $\rg(\mc{X})\subseteq\rg(\mc{Y}\,\s\,\mc{D})$ since 
$\mc{X}=\mc{D}^{\dagger,\D}={\color{blue}\mc{D}^\dagger \,\s\,\mc{D}\,\s\,\mc{D}^\D=}\mc{Y}\,\s\,\mc{D}\,\s\,\mc{D}^\D$.
Now if $\mc{Q}\in \nl(\mc{D}\,\s\,\mc{Y})=\nl(\mc{D}\,\s\,\mc{D}^{\dagger})\subseteq \nl(\mc{D}^{\dagger})\subseteq \nl(\mc{D}^{k})$ then $\mc{D}^k\,\s\,\mc{Q}=0$.
Further,
\begin{equation*}
    \mc{X}\,\s\,\mc{Q}=\mc{D}^{\dagger}\,\s\,\mc{D}\,\s\,\mc{D}^\D\,\s\,\mc{Q}=\mc{D}^{\dagger}\,\s\,(\mc{D}^\D)^k\,\s\,\mc{D}^k\,\s\,\mc{Q}=0.
\end{equation*}
Thus $\nl(\mc{D}\,\s\,\mc{Y})\subseteq \nl(\mc{X})$.\\
(ii) Similar to the part (i).
\end{proof}

The above theorem can be restated as the following remark.
\begin{remark}
  For $\mathrm{ind}(\mc{D})=k$ it follows
  \begin{enumerate}
      \item[\rm (i)] $\mc{D}^{\dagger,\D}\,\s\,\mc{D}\,\s\,\mc{D}^{\dagger}=\mc{D}^{\dagger}\,\s\,\mc{D}\,\s\,\mc{D}^{\dagger,\D}\Longleftrightarrow \nl(\mc{D}^{\dagger})\subseteq \nl(\mc{D}^{k})$.
      \item[\rm (ii)]  $\mc{D}^{\dagger}\,\s\,\mc{D}\,\s\,\mc{D}^{\D,\dagger}=\mc{D}^{\D,\dagger}\,\s\,\mc{D}\,\s\,\mc{D}^{\dagger}\Longleftrightarrow \rg(\mc{D}^k)\subseteq \rg(\mc{D}^{\dagger})$.
  \end{enumerate}
\end{remark}
\begin{theorem}
 For 
  $\mathrm{ind}(\mc{D})=k$ it follows
  \begin{equation*}
      \mc{D}^{\ep,\dagger}=\mc{D}^{\dagger,\ep} \Longleftrightarrow \nl(\mc{D}^{\dagger})\subseteq \nl(\mc{D}^{\ep}) \wedge\rg(\mc{D}^{\ep})\subseteq \rg(\mc{D}^{\dagger}\,\s\,\mc{D}^k).
  \end{equation*}
\end{theorem}
\begin{proof}
    Let $\mc{D}^{\ep,\dagger}=\mc{D}^{\ep}\,\s\,\mc{D}\,\s\,\mc{D}^{\dagger}=\mc{D}^{\dagger}\,\s\,\mc{D}\,\s\,\mc{D}^{\ep}=\mc{D}^{\dagger,\ep}$.
    From the identities
    \begin{equation*}
    \aligned
        \mc{D}^{\ep}&=\mc{D}^{\ep}\,\s\,\mc{D}\,\s\,\mc{D}^{\dagger}\,\s\,\mc{D}\,\s\, \mc{D}^{\ep}= \mc{D}^{\dagger}\,\s\,\mc{D}\,\s\, \mc{D}^{\ep}=\mc{D}^{\dagger}\,\s\,\mc{D}^k\,\s\, (\mc{D}^{\ep})^k,\\
        \mc{D}^{\ep}&=\mc{D}^{\ep}\,\s\,\mc{D}\,\s\,\mc{D}^{\dagger}\,\s\,\mc{D}\,\s\, \mc{D}^{\ep}=\mc{D}^{\ep}\,\s\,\mc{D}\,\s\, \mc{D}^{\dagger},
        \endaligned
    \end{equation*}
    we obtain $\nl(\mc{D}^{\dagger})\subseteq \nl(\mc{D}^{\ep}) \mbox{ and }\rg(\mc{D}^{\ep})\subseteq \rg(\mc{D}^{\dagger}\,\s\,\mc{D}^k)$.\\
    For the converse statement, in view of Theorem \ref{thm3.5}, it is sufficient to show $\nl(\mc{D}\,\s\,\mc{D}^{\dagger})\subseteq \nl(\mc{D}^{\ep})$ and $\rg(\mc{D}^{\ep})\subseteq\rg(\mc{D}^{\dagger}\,\s\,\mc{D})$. Clearly,  $\rg(\mc{D}^{\ep})\subseteq \rg(\mc{D}^{\dagger}\,\s\,\mc{D}^k)= {\rg(\mc{D}^{\dagger}\,\s\,\mc{D}\,\s\,\mc{D}^k\,\s\,\mc{D}^\D)\subseteq\rg(\mc{D}^{\dagger}\,\s\,\mc{D})}$.
The given null condition $\nl(\mc{D}^{\dagger})\subseteq \nl(\mc{D}^{\ep})$ implies $\mc{D}^{\ep}=\mc{Q}\,\s\,\mc{D}^{\dagger}$ for some $\mc{Q}\in \mathbb C^{\textbf{N}(s)\times \textbf{N}(s)}$.
Finally
    \begin{equation*}
       \mc{D}^{\ep} =\mc{Q}\,\s\,\mc{D}^{\dagger}=\mc{Q}\,\s\,\mc{D}^{\dagger}\,\s\,\mc{D}\,\s\,\mc{D}^{\dagger}=\mc{D}^{\ep}\,\s\,\mc{D}\,\s\,\mc{D}^{\dagger}
    \end{equation*}
confirms $\nl(\mc{D}\,\s\,\mc{D}^{\dagger})\subseteq \nl(\mc{D}^{\ep})$.
\end{proof}

\subsection{Properties of CMP, DMP, MPD, MPCEP, and CEPMP inverses}\label{SubSecCGI}
A few characterizations of these composite inverses are discussed in the subsequent theorems.   {
\begin{theorem} \label{thm2.66}
   Let   $\mathrm{ind}(\mc{D})=k$. Then  $\mc{D}^{c,\dagger}$ is the unique solution to equations
   \begin{equation}\label{eqqq2.3}
       \mc{Z}\,\s\,\mc{D}\,\s\,\mc{Z}=\mc{Z},~  \mc{D}\,\s\,\mc{Z}=\mc{D}\,\s\,\mc{D}^\D\,\s\,\mc{D}\,\s\,\mc{D}^{\dagger},~\mc{Z}\,\s\,\mc{D}=\mc{D}^{\dagger}\,\s\,\mc{D}\,\s\,\mc{D}^\D\,\s\,\mc{D}.
   \end{equation}
\end{theorem}
\begin{proof}
The proof is follows from Theorem \ref{thm3.2} by specifying $\mc{X}=\mc{D}^{D,\dagger}$ and $\mc{Y}=\mc{D}^{\dagger}$.
\end{proof}
\begin{theorem} \label{thm2.77}
The DMP inverse $\mc{D}^{\D,\dagger}$ is the unique solution to tensor equations
   \begin{equation}\label{eqqq2.4}
       \mc{Z}\,\s\,\mc{D}\,\s\,\mc{Z}=\mc{Z},~  \mc{D}\,\s\,\mc{Z}=\mc{D}\,\s\,\mc{D}^\D\,\s\,\mc{D}\,\s\,\mc{D}^{\dagger},~\mc{Z}\,\s\,\mc{D}=\mc{D}^\D\,\s\,\mc{D}.
   \end{equation}
\end{theorem}
\begin{proof}
The proof is follows from Theorem \ref{thm3.4} by specifying $\mc{X}=\mc{D}^{D}$ and $\mc{Y}=\mc{D}^{\dagger}$.
\end{proof}
\begin{corollary} \label{cor2.33}
The MPD inverse  $\mc{D}^{\dagger,\D}$ is the unique solution to the tensor equations
   \begin{equation*}
       \mc{Z}\,\s\,\mc{D}\,\s\,\mc{Z}=\mc{Z},~  \mc{Z}\,\s\,\mc{D}=\mc{D}^{\dagger}\,\s\,\mc{D}\,\s\,\mc{D}^\D\,\s\,\mc{D},~\mc{D}\,\s\,\mc{Z}=\mc{D}\,\s\,\mc{D}^\D.
   \end{equation*}
\end{corollary}
\begin{proof}
 The proof follows from Theorem \ref{thm3.2} by specifying $\mc{X}=\mc{D}^{D}$ and $\mc{Y}=\mc{D}^{\dagger}$.
\end{proof}
An alternative proof for the Theorem \ref{thm2.66}, Theorem \ref{thm2.77} and Corollary \ref{cor2.33} can be found in \cite{PublishedSalemi23}.
}
\begin{theorem}
Let  $\mathrm{ind}(\mc{D})=k$. Then
     \begin{enumerate}
         \item[\rm (i)] $\mc{D}^{\D,\dagger}\,\s\,\mc{D}=\mc{D}^{c,\dagger}\,\s\,\mc{D}\Longleftrightarrow\rg(\mc{D}^k)\subseteq\rg(\mc{D}^{\dagger})$.
         \item[\rm (ii)] $\mc{D}\,\s\,\mc{D}^{\dagger,\D}=\mc{D}\,\s\,\mc{D}^{c,\dagger}\Longleftrightarrow\mc{D}^k=\mc{D}^{k+1}\,\s\,\mc{D}^{\dagger}$.
     \end{enumerate}
\end{theorem}
\begin{proof}
 (i) Consider  $\mc{D}^{\D,\dagger}\,\s\,\mc{D}=\mc{D}^{c,\dagger}\,\s\,\mc{D}$.
 Then $\mc{D}^\D\,\s\,\mc{D}=\mc{D}^{\dagger}\,\s\,\mc{D}\,\s\,\mc{D}^\D\,\s\,\mc{D}$.
 In addition, $\mc{D}^k=\mc{D}^\D\,\s\,\mc{D}\,\s\,\mc{D}^k=\mc{D}^{\dagger}\,\s\,\mc{D}\,\s\,\mc{D}^\D\,\s\,\mc{D}\,\s\,\mc{D}^k$.
 Thus $\rg(\mc{D}^k)\subseteq\rg(\mc{D}^{\dagger})$.\\
 Conversely, let $\rg(\mc{D}^k)\subseteq\rg(\mc{D}^{\dagger})$.
 Then there exist a tensor $\mc{T}\in \mathbb C^{\textbf{N}(s)\times \textbf{N}(s)}$ such that $\mc{D}^k=\mc{D}^{\dagger}\,\s\,\mc{T}$.
 Now
  {
\begin{equation*}
\aligned
\mc{D}^{\D,\dagger}\,\s\,\mc{D}&=\mc{D}^{\D}\,\s\,\mc{D}=\mc{D}^{k}\,\s\,(\mc{D}^\D)^k=\mc{D}^{\dagger}\,\s\,\mc{T}\,\s\,(\mc{D}^\D)^k\\
& =\mc{D}^{\dagger}\,\s\,\mc{D}\,\s\,\mc{D}^{\dagger}\,\s\,\mc{T}\,\s\,(\mc{D}^\D)^k=\mc{D}^{\dagger}\,\s\,\mc{D}\,\s\,\mc{D}^{k}\,\s\,(\mc{D}^\D)^k\\
&=\mc{D}^{\dagger}\,\s\,\mc{D}\,\s\,\mc{D}\,\s\,\mc{D}^{\D}=\mc{D}^{\dagger}\,\s\,\mc{D}\,\s\,\mc{D}^D\,\s\,\mc{D}\,\s\,\mc{D}^{\dagger}\,\s\,\mc{D}\\
&=\mc{D}^{c,\dagger}\,\s\,\mc{D}.
\endaligned
\end{equation*}}
(ii) It will follow by using the similar arguments of the part (i).
\end{proof}

The characterizations of MPCEP and CEPMP inverse for tensors are discussed in the next theorems.

\begin{theorem}\label{thmm2.3}
The {\rm MPCEP} inverse $\mc{D}^{\dagger,\ep}$ is the unique solution to equations
   \begin{equation}\label{eqqq2.5}
       \mc{Z}\,\s\,\mc{D}\,\s\,\mc{Z}=\mc{Z},~  \mc{D}\,\s\,\mc{Z}=\mc{D}\,\s\,\mc{D}^{\ep},~\mc{Z}\,\s\,\mc{D}=\mc{D}^{\dagger}\,\s\,\mc{D}\,\s\,\mc{D}^{\ep}\,\s\,\mc{D}.
   \end{equation}
\end{theorem}
\begin{proof}
    Let $\mc{Z}=\mc{D}^{\dagger,\ep}=\mc{D}^{\dagger}\,\s\,\mc{D}\,\s\,\mc{D}^{\ep}$. Then $\mc{D}\,\s\,\mc{Z}=\mc{D}\,\s\,\mc{D}^{\ep}$,  $\mc{D}\,\s\,\mc{Z}=\mc{D}\,\s\,\mc{D}^{\ep}$, and
    \begin{equation*}   \mc{Z}\,\s\,\mc{D}\,\s\,\mc{Z}=\mc{Z}\,\s\,\mc{D}\,\s\,\mc{D}^{\ep}=\mc{D}^{\dagger}\,\s\,\mc{D}\,\s\,\mc{D}^{\ep}\,\s\,\mc{D}\,\s\,\mc{D}^{\ep}=\mc{D}^{\dagger}\,\s\,\mc{D}\,\s\,\mc{D}^{\ep}=\mc{Z}
    \end{equation*}
If there are two solutions say $\mc{Z}_1$ and $\mc{Z}_2$, which satisfy \eqref{eqqq2.5}, then
\begin{equation*}
\aligned
\mc{Z}_1&=\mc{Z}_1\,\s\,\mc{D}\,\s\,\mc{Z}_1=\mc{Z}_1\,\s\,\mc{D}\,\s\,\mc{D}^{\ep}=\mc{Z}_1\,\s\,\mc{D}\,\s\,\mc{Z}_2=\mc{D}^{\dagger}\,\s\,\mc{D}\,\s\,\mc{D}^{\ep}\,\s\,\mc{D}\,\s\,\mc{Z}_2\\
&=\mc{Z}_2\,\s\,\mc{D}\,\s\,\mc{Z}_2=\mc{Z}_2.
\endaligned
\end{equation*}
\end{proof}

Corollary \ref{Cor24T} is verified using similar principles as in the proof of Theorem \ref{thmm2.3}.
\begin{corollary}\label{Cor24T}
The {\rm CEPMP} inverse $\mc{D}^{\ep,\dagger}$ is the unique solution of the following tensor equations:
   \begin{equation*}
       \mc{Z}\,\s\,\mc{D}\,\s\,\mc{Z}=\mc{Z},~  \mc{D}\,\s\,\mc{Z}=\mc{D}\,\s\,\mc{D}^{\ep}\,\s\,\mc{D}\,\s\,\mc{D}^{\dagger},~\mc{Z}\,\s\,\mc{D}=\mc{D}^{\ep}\,\s\,\mc{D}.
   \end{equation*}
\end{corollary}

In view of Lemma \ref{lm4.2}, we obtain the following  representation for MPCEP and CEPMP inverse.
\begin{lemma}\label{lmm2.15}
  Let 
  $\mathrm{ind}(\mc{D})=k$.
  Then for any positive integer $l \geq k$ it follows
  \begin{enumerate}
      \item[\rm (i)] $\mc{D}^{\dagger,\ep}=\mc{D}^{\dagger}\,\s\,\mc{D}^l\,\s\,(\mc{D}^l)^{\dagger}$.
      \item[\rm (ii)] $\mc{D}^{\ep,\dagger}=\mc{D}^{\D}\,\s\,\mc{D}^l\,\s\,(\mc{D}^l)^{\dagger}$.
  \end{enumerate}
\end{lemma}
\begin{proof}
     Using $\mc{D}^{\ep}=\mc{D}^{\D}\,\s\,\mc{D}^l\,\s\,(\mc{D}^l)^{\dagger}$ it is obtained
    \begin{equation*}        \mc{D}^{\dagger,\ep}=\mc{D}^{\dagger}\,\s\,\mc{D}\,\s\,\mc{D}^{\D}\,\s\,\mc{D}^l\,\s\,(\mc{D}^l)^{\dagger}=\mc{D}^{\dagger}\,\s\,\mc{D}^l\,\s\,(\mc{D}^l)^{\dagger},
    \end{equation*}
and
   \begin{equation*}
\aligned
\mc{D}^{\ep,\dagger}&=\mc{D}^{\D}\,\s\,\mc{D}^l\,\s\,(\mc{D}^l)^{\dagger}\,\s\,\mc{D}\,\s\,\mc{D}^{\dagger}=\mc{D}^{\D}\,\s\,(\mc{D}^l\,\s\,(\mc{D}^l)^{\dagger})^*\,\s\,(\mc{D}\,\s\,\mc{D}^{\dagger})^*\\   &=\mc{D}^{\D}\,\s\,(\mc{D}\,\s\,\mc{D}^{\dagger}\,\s\,\mc{D}^l\,\s\,(\mc{D}^l)^{\dagger})^*=\mc{D}^{\D}\,\s\,(\mc{D}^l\,\s\,(\mc{D}^l)^{\dagger})^*=\mc{D}^{\D}\,\s\,\mc{D}^l\,\s\,(\mc{D}^l)^{\dagger}.
\endaligned
   \end{equation*}
\end{proof}
\begin{theorem}
 For $\mathrm{ind}(\mc{D})=k$ it follows
  \begin{enumerate}
      \item[\rm (i)] $\mc{D}^{\dagger,\ep}=\mc{D}^{(2)}_{\rg(\mc{D}^{\dagger}\,\s\,\mc{D}^k),\nl((\mc{D}^k)^\dagger)}=\mc{D}^{(2)}_{\rg(\mc{D}^{\dagger}\,\s\,\mc{D}^k),\nl((\mc{D}^k)^*)}$.
      \item[\rm (ii)]  $\mc{D}^{\ep,\dagger}=\mc{D}^{(2)}_{\rg(\mc{D}^{k}),\nl((\mc{D}^k)^\dagger)}=\mc{D}^{(2)}_{\rg(\mc{D}^{k}),\nl((\mc{D}^k)^*)}$.
  \end{enumerate}
\end{theorem}
\begin{proof}
    (i) Let $\mc{Z}=\mc{D}^{\dagger,\ep}=\mc{D}^{\dagger}\,\s\,\mc{D}\,\s\,\mc{D}^{\ep}$. Then it follows
    \begin{equation*}      \mc{Z}\,\s\,\mc{D}\,\s\,\mc{Z}=\mc{D}^{\dagger}\,\s\,\mc{D}\,\s\,\mc{D}^{\ep}\,\s\,\mc{D}\,\s\,\mc{D}^{\dagger}\,\s\,\mc{D}\,\s\,\mc{D}^{\ep}=\mc{D}^{\dagger}\,\s\,\mc{D}\,\s\,\mc{D}^{\ep}=\mc{Z}.
    \end{equation*}
From $\mc{Z}=\mc{D}^{\dagger}\,\s\,\mc{D}\,\s\,\mc{D}^{\ep}=\mc{D}^{\dagger}\,\s\,(\mc{D})^k\,\s\,(\mc{D}^{\ep})^k$ and $\mc{D}^{\dagger}\,\s\,\mc{D}^k=\mc{D}^{\dagger}\,\s\,(\mc{D})^k\,\s\,(\mc{D}^k)^{\dagger}\,\s\,\mc{D}^k=\mc{D}^{\dagger,\ep}\,\s\,\mc{D}^k$, we obtain $\rg(\mc{Z})=\rg(\mc{D}^{\dagger}\,\s\,\mc{D}^k)$. By Lemma \ref{lmm2.15} (i), we have
\begin{equation*}
    \nl(\mc{Z})=\nl(\mc{D}^{\dagger}\,\s\,\mc{D}^k\,\s\,(\mc{D}^k)^{\dagger})\supseteq \nl((\mc{D}^k)^{\dagger}), \mbox{ and}
\end{equation*}
$\nl(\mc{Z})\subseteq \nl((\mc{D}^k)^{\dagger})$ is follows from the below identity:
\begin{equation*}
\aligned
    (\mc{D}^k)^{\dagger}&=(\mc{D}^k)^{\dagger}\,\s\,\mc{D}^k\,\s\,(\mc{D}^k)^{\dagger}=(\mc{D}^k)^{\dagger}\,\s\,\mc{D}\,\s\,\mc{D}^\D\,\s\,\mc{D}^k\,\s\,(\mc{D}^k)^{\dagger}=(\mc{D}^k)^{\dagger}\,\s\,\mc{D}\,\s\,\mc{D}^{\ep}\\
    &=(\mc{D}^k)^{\dagger}\,\s\,\mc{D}\,\s\,\mc{D}^{\dagger,\ep}=(\mc{D}^k)^{\dagger}\,\s\,\mc{D}\,\s\,\mc{Z}.
    \endaligned
\end{equation*}
It completes the proof due to the fact that $\nl((\mc{D}^k)^{\dagger})=\nl((\mc{D}^k)^{*})$.\\
(ii) The proof is similar to the part (i).
\end{proof}

\section{Applications in solving multilinear system }\label{SecSMS}

In this section, we will discuss the application of composite generalized inverses for a suitable inner and outer inverse. In the first three results, we chose the inner inverse as $A^{\dagger}$ and the outer inverse as $A^{\D,\dagger}$ as stated below.

\begin{theorem}
   For $k=\mathrm{ind}(\mc{D})$  the general solution to the multilinear system
\begin{equation}\label{eqq4.1}
    \mc{D}^k\,\s\,\mc{Z}=\mc{D}^k\,\s\,\mc{D}^\dagger\,\s\,\mc{B}, ~~ \mc{B} \in  \mathbb C^{\textbf{N}(s)}
\end{equation}
    is expressed as
    \begin{equation}\label{eqq4.2}
        \mc{Z} = \mc{D}^{c,\dagger}\,\s\,\mc{B} +(\mc{I} - \mc{D}^{c,\dagger}\,\s\,\mc{D})\,\s\,\mc{Q}, \text{~where~}  \mc{Q} \in  \mathbb C^{\textbf{N}(s)} \mbox{ is arbitray}.
    \end{equation}
\end{theorem}
\begin{proof}
 Let $\mc{Z} = \mc{D}^{c,\dagger}\,\s\,\mc{B} +(\mc{I} - \mc{D}^{c,\dagger}\,\s\,\mc{D})\,\s\,\mc{Q}$. Then
 \begin{equation*}
     \mc{D}^k\,\s\,\mc{Z}=  \mc{D}^k\,\s\,\mc{D}^{c,\dagger}\,\s\,\mc{B}+(\mc{D}^k-\mc{D}^k\,\s\,\mc{D}^{c,\dagger}\,\s\,\mc{D})\,\s\,\mc{Q}=\mc{D}^k\,\s\,\mc{D}^{c,\dagger}\,\s\,\mc{B}=\mc{D}^k\,\s\,\mc{D}^{\dagger}\,\s\,\mc{B}.
 \end{equation*}
 Thus $\mc{Z}$ satisfies the equation \eqref{eqq4.1}.
 On other hand, if $\mc{Z}_1$ be any other solution of \eqref{eqq4.1}, then
 \begin{equation*}
 \aligned
 \mc{Z}_1&=\mc{D}^{c,\dagger}\,\s\,\mc{D}\,\s\,\mc{Z}_1+\mc{Z}_1-\mc{D}^{c,\dagger}\,\s\,\mc{D}\,\s\,\mc{Z}_1=\mc{D}^{\dagger}\,\s\,\mc{D}\,\s\,\mc{D}^\D\,\s\,\mc{D}\,\s\,\mc{Z}_1+\mc{Z}_1-\mc{D}^{c,\dagger}\,\s\,\mc{D}\,\s\,\mc{Z}_1\\
&=\mc{D}^{\dagger}\,\s\,\mc{D}\,\s\,(\mc{D}^\D)^k\,\s\,\mc{D}^k\,\s\,\mc{Z}_1+\mc{Z}_1-\mc{D}^{c,\dagger}\,\s\,\mc{D}\,\s\,\mc{Z}_1\\
&=\mc{D}^{\dagger}\,\s\,\mc{D}\,\s\,(\mc{D}^\D)^k\,\s\,\mc{D}^k\,\s\,\mc{D}^{\dagger}\,\s\,\mc{B}+\mc{Z}_1-\mc{D}^{c,\dagger}\,\s\,\mc{D}\,\s\,\mc{Z}_1\\
&=\mc{D}^{\dagger}\,\s\,\mc{D}\,\s\,\mc{D}^\D\,\s\,\mc{D}\,\s\,\mc{D}^{\dagger}\,\s\,\mc{B}+\mc{Z}_1-\mc{D}^{c,\dagger}\,\s\,\mc{D}\,\s\,\mc{Z}_1\\
&=\mc{D}^{c,\dagger}\,\s\,\mc{B}+\mc{Z}_1-\mc{D}^{c,\dagger}\,\s\,\mc{D}\,\s\,\mc{Z}_1=\mc{D}^{c,\dagger}\,\s\,\mc{B}+(\mc{I}-\mc{D}^{c,\dagger}\,\s\,\mc{D})\,\s\,\mc{Z}_1.
\endaligned
 \end{equation*}
 Thus $\mc{Z}_1$ is of the form \eqref{eqq4.2}, which completes the verification.
\end{proof}
\begin{theorem}
    Let 
    $k=\mathrm{ind}(\mc{D})$.
    Then  {$\mc{D}^{c,\dagger}\,\s\,\mc{B}$ } is the unique solution to the constrained multilinear system
\begin{equation}\label{eqqq4.3}
    \mc{D}\,\s\,\mc{Z}=\mc{B}, ~~ \mc{B} \in  \rg(\mc{D}^k),
\end{equation}
in the range space $\rg(\mc{D}^{\dagger}\,\s\,\mc{D}^k)$.
\end{theorem}
\begin{proof}
First observe that $\rg(\mc{D}^k)=\rg(\mc{D}\,\s\,\mc{D}^{c,\dagger})$ is follows from
\begin{equation*}
    \mc{D}^k=\mc{D}\,\s\,\mc{D}^{c,\dagger}\,\s\,\mc{D}^k \mbox{ and } \mc{D}\,\s\,\mc{D}^{c,\dagger}=\mc{D}^k\,\s\,(\mc{D}^\D)^k\,\s\,\mc{D}\,\s\,\mc{D}^\dagger.
\end{equation*}
Let $ \mc{B} \in  \rg(\mc{D}^k)=\rg(\mc{D}\,\s\,\mc{D}^{c,\dagger})$. Then we obtain $\mc{B}=\mc{D}\,\s\,\mc{D}^{c,\dagger}\,\s\,\mc{D}$ for some $\mc{D}\in \mathbb C^{\textbf{N}(s)\times \textbf{N}(s)}$.  Now
\begin{equation*}
\aligned
    \mc{D}\,\s\,\mc{D}^{c,\dagger}\,\s\,\mc{B}&=\mc{D}\,\s\,\mc{D}^{c,\dagger}\,\s\,\mc{D}\,\s\,\mc{D}^{c,\dagger}\,\s\,\mc{D}\\
    &=\mc{D}\,\s\,\mc{D}^\D\,\s\,\mc{D}\,\s\,\mc{D}^{\dagger}\,\s\,\mc{D}=\mc{D}\,\s\,\mc{D}^{c,\dagger}\,\s\,\mc{D}\\
    &=\mc{B},
    \endaligned
\end{equation*}
and $\mc{D}^{c,\dagger}\,\s\,\mc{B}=\mc{D}^{\dagger}\,\s\,\mc{D}^k\,\s\,(\mc{D}^\D)^k\,\s\,\mc{D}\,\s\,\mc{D}^{\dagger}\,\s\,\mc{B}$.
Thus $\mc{D}^{c,\dagger}\,\s\,\mc{B}\in \rg(\mc{D}^{\dagger}\,\s\,\mc{D}^k)$, and satisfying the equation \eqref{eqqq4.3}.
To show the uniqueness of the solution, let $\mc{Z}_1,\mc{Z}_2\in\rg(\mc{D}^{\dagger}\,\s\,\mc{D}^k)$ be two solution of \eqref{eqqq4.3}.
Then
\begin{equation*}
    \mc{Z}_1-\mc{Z}_2\in \nl(\mc{D})\cap\rg(\mc{D}^{\dagger}\,\s\,\mc{D}^k) \subseteq \nl(\mc{D}^{c,\dagger}\,\s\,\mc{D})\cap \rg(\mc{D}^{c,\dagger}\,\s\,\mc{D}) = \{0\}.
\end{equation*}
Hence $\mc{D}^{c,\dagger}\,\s\,\mc{B}$ is  {the unique solution \eqref{eqqq4.3} in $\rg(\mc{D}^{\dagger}\,\s\,\mc{D}^k)$.}
\end{proof}

\begin{theorem}
The general solution to the multilinear system
  \begin{equation}\label{eqq4.6}    \mc{D}\,\s\,\mc{Z}=\mc{D}\,\s\,\mc{D}^{c,\dagger}\,\s\,\mc{B}, \ \  \mc{B} \in  \mathbb C^{\textbf{N}(n)},\ k=\mathrm{ind}(\mc{D})
\end{equation}
is given by
\begin{equation}\label{eqq4.7}
        \mc{Z} = \mc{D}^{c,\dagger}\,\s\,\mc{B} +(\mc{I} - \mc{D}^{\dagger}\,\s\,\mc{D})\,\s\,\mc{Q}, \mbox{ where } \mc{Q} \in  \mathbb C^{\textbf{N}(s)} \mbox{ is arbtrary}.
    \end{equation}
\end{theorem}
\begin{proof}
    Let  $\mc{Z} = \mc{D}^{c,\dagger}\,\s\,\mc{B} +(\mc{I} - \mc{D}^{\dagger}\,\s\,\mc{D})\,\s\,\mc{Q}$. Then
    \begin{equation*}
       \mc{D}\,\s\,\mc{Z}= \mc{D}\,\s\, \mc{D}^{c,\dagger}\,\s\,\mc{B}+(\mc{D}-\mc{D}\,\s\,\mc{D}^{\dagger}\,\s\,\mc{D})\,\s\,\mc{Q}=\mc{D}\,\s\, \mc{D}^{c,\dagger}\,\s\,\mc{B}.
    \end{equation*}
    If $\mc{Z}_1$ is any other solution of \eqref{eqq4.6}, then $\mc{D}\,\s\,\mc{Z}_1=\mc{D}\,\s\,\mc{D}^{c,\dagger}\,\s\,\mc{B}$. Now
\begin{equation*}
\aligned
\mc{Z}_1&=\mc{D}^{\dagger}\,\s\,\mc{D}\,\s\,\mc{Z}_1+\mc{Z}_1-\mc{D}^{\dagger}\,\s\,\mc{D}\,\s\,\mc{Z}_1=\mc{D}^{\dagger}\,\s\,\mc{D}\,\s\,\mc{D}^{c,\dagger}\,\s\,\mc{B}+\mc{Z}_1-\mc{D}^{\dagger}\,\s\,\mc{D}\,\s\,\mc{Z}_1\\
&=\mc{D}^{c,\dagger}\,\s\,\mc{B}+\mc{Z}_1-\mc{D}^{\dagger}\,\s\,\mc{D}\,\s\,\mc{Z}_1=\mc{D}^{c,\dagger}\,\s\,\mc{B}+(\mc{I}-\mc{D}^{\dagger}\,\s\,\mc{D})\,\s\,\mc{Z}_1.
\endaligned
\end{equation*}
Thus $\mc{Z}_1$ is of the form \eqref{eqq4.7}, and hence $\mc{Z}$ is the general of \eqref{eqq4.6}.
\end{proof}
In the next result, we choose $\mc{D}^{\dagger}\in \mc{D}\{1\}$ and $\mc{D}^{\D}\in \mc{D}\{2\}$. \begin{theorem}
    Let 
    $k=\mathrm{ind}(\mc{D})$.
    Then $\mc{D}^{\D,\dagger}\,\s\,\mc{B}$ is the unique solution to the multilinear system
\begin{equation}\label{eqq4.3}
    \mc{D}\,\s\,\mc{Z}=\mc{B}, ~~ \mc{B} \in  \rg(\mc{D}^k),
\end{equation}
in the range space $\rg(\mc{D}^{\D}\,\s\,\mc{D}^k)$.
\end{theorem}
\begin{proof}
   Let $ \mc{B} \in  \rg(\mc{D}^k)$. Then $\mc{B}=\mc{D}^{k}\,\s\,\mc{Q}$ for some $\mc{Q}\in \mathbb C^{\textbf{N}(s)\times \textbf{N}(s)}$.
Now
\begin{equation*}
\mc{D}\,\s\,\mc{D}^{\D,\dagger}\,\s\,\mc{B}=\mc{D}\,\s\,\mc{D}^{\D}\,\s\,\mc{D}\,\s\,\mc{D}^{\dagger}\,\s\,\mc{D}^k\,\s\,\mc{Q}=\mc{D}^\D\,\s\,\mc{D}^{k+1}\,\s\,\mc{Q}=\mc{D}^{k}\,\s\,\mc{Q} =\mc{B}.
\end{equation*}
If $\mc{Z}_1,\mc{Z}_2\in\rg(\mc{D}^{\D}\,\s\,\mc{D}^k)$ are different solutions to \eqref{eqq4.3}, then
\begin{equation*}
    \mc{Z}_1-\mc{Z}_2\in \nl(\mc{D})\cap\rg(\mc{D}^{\D}\,\s\,\mc{D}^k) \subseteq \nl(\mc{D}^{\D,\dagger}\,\s\,\mc{D})\cap \rg(\mc{D}^{\D,\dagger}\,\s\,\mc{D}) = \{0\}.
\end{equation*}
Hence $\mc{D}^{c,\dagger}\,\s\,\mc{B}$ is the unique solver of \eqref{eqq4.3} in $\rg(\mc{D}^{\D}\,\s\,\mc{D}^k)$.
\end{proof}
\begin{corollary}
    {Let $\mathrm{ind}(\mc{D})= k $.}
    Then $A^{\dagger,\D}\,\s\,\mc{B}$ is the unique solution to the multilinear system
\begin{equation*}
    \mc{D}\,\s\,\mc{Z}=\mc{B}, ~~ \mc{B} \in  \rg(\mc{D}^k),
\end{equation*}
in the range space $\rg(\mc{D}^{\dagger}\,\s\,\mc{D}^k)$.
\end{corollary}
\begin{theorem}\label{thmm3.6}
    { Let 
    $\mathrm{ind}(\mc{D})=k$.} Then  the tensor equation
    \begin{equation*}        \mc{D}\,\s\,\mc{Z}=\mc{D}^k\,\s\,(\mc{D}^k)^{\dagger}\,\s\,\mc{B},\ \ \mc{B}\in \mathbb C^{\textbf{N}(s)}
    \end{equation*}
  is solvable, and its general solution is
   \begin{equation}\label{eqqq3.7}
        \mc{Z} = \mc{D}^{\dagger,\ep}\,\s\,\mc{B} +(\mc{I} - \mc{D}^{\dagger}\,\s\,\mc{D})\,\s\,\mc{Q}, \text{~where~}  \mc{Q} \in  \mathbb C^{\textbf{N}(s)} \mbox{ is arbitray}.
    \end{equation}
\end{theorem}
\begin{proof}
 Consider $ \mc{Z} = \mc{D}^{\dagger,\ep}\,\s\,\mc{B} +(\mc{I} - \mc{D}^{\dagger}\,\s\,\mc{D})\,\s\,\mc{Q}$.
 Now an application of Lemma \ref{lm4.2} (iii) leads to
 \begin{equation*}
 \aligned
\mc{D}\,\s\,\mc{Z}&= {\mc{D}\,\s\,\mc{D}^{\dagger,\ep}\,\s\,\mc{B}}=\mc{D}\,\s\,\mc{D}^{\dagger}\,\s\,\mc{D}\,\s\,\mc{D}^{\ep}\,\s\,\mc{B}  \\ &=\mc{D}\,\s\,\mc{D}^{\ep}\,\s\,\mc{B}=\mc{D}\,\s\,\mc{D}^\D\,\s\,\mc{D}^k\,\s\,(\mc{D}^k)^{\dagger}\,\s\,\mc{B}\\
&=\mc{D}^k\,\s\,(\mc{D}^k)^{\dagger}\,\s\,\mc{B}.
\endaligned
 \end{equation*}
 If $\mc{Z}_1$ be any solution of \eqref{eqqq3.7}, then
 \begin{equation*}
 \aligned
\mc{Z}_1&=\mc{D}^{\dagger}\,\s\,\mc{D}\,\s\,\mc{Z}_1+\mc{Z}_1-\mc{D}^{\dagger}\,\s\,\mc{D}\,\s\,\mc{Z}_1=\mc{D}^{\dagger}\,\s\,\mc{D}^k\,\s\,(\mc{D}^k)^{\dagger}\,\s\,\mc{B}+(\mc{I}-\mc{D}^{\dagger}\,\s\,\mc{D})\,\s\,\mc{Z}_1\\  &=\mc{D}^{\dagger}\,\s\,\mc{D}\,\s\,\mc{D}^\D\,\s\,\mc{D}^k\,\s\,(\mc{D}^k)^{\dagger}\,\s\,\mc{B}+(\mc{I}-\mc{D}^{\dagger}\,\s\,\mc{D})\,\s\,\mc{Z}_1\\
&=\mc{D}^{\dagger}\,\s\,\mc{D}\,\s\,\mc{D}^{\ep}\,\s\,\mc{B}+(\mc{I}-\mc{D}^{\dagger}\,\s\,\mc{D})\,\s\,\mc{Z}_1=\mc{D}^{\dagger,\ep}\,\s\,\mc{B}+(\mc{I}-\mc{D}^{\dagger}\,\s\,\mc{D})\,\s\,\mc{Z}_1.
\endaligned
 \end{equation*}
 Hence $\mc{Z}_1$ is of the form \eqref{eqqq3.7} and completes the verification.
\end{proof}
\begin{theorem}\label{thmm3.7}
    The constrained tensor equation
    \begin{equation*}
    \mc{D}\,\s\,\mc{Z}=\mc{B},\ \ \mc{B}\in \rg(\mc{D}^k),\ \ k=\mathrm{ind}(\mc{D})
    \end{equation*}
  is solvable, and its general solution is
   \begin{equation}\label{eqqq3.8}
        \mc{Z} = \mc{D}^{\dagger,\ep}\,\s\,\mc{B} +(\mc{I} - \mc{D}^{\dagger}\,\s\,\mc{D})\,\s\,\mc{Q}, \text{~where~}  \mc{Q} \in  \mathbb C^{\textbf{N}(s)} \mbox{ is arbitray}.
    \end{equation}
\end{theorem}
\begin{proof}
    Let $\mc{B}\in \rg(\mc{D}^k)$. Then $\mc{B}=\mc{D}^k\,\s\,\mc{S}$ for some $\mc{S} \in  \mathbb C^{\textbf{N}(n)}$. Now
    \begin{equation*}
        \mc{D}\,\s\,\mc{Z}= \mc{D}\,\s\,\mc{D}^{\dagger}\,\s\,\mc{D}\,\s\,\mc{D}^{\ep}\,\s\,\mc{B}=\mc{D}\,\s\,\mc{D}^{\ep}\,\s\,\mc{D}^k\,\s\,\mc{S}=\mc{D}\,\s\,\mc{D}^{\D}\,\s\,\mc{D}^k\,\s\,\mc{S}=\mc{B}.
    \end{equation*}
 If $\mc{Z}_1$ be any solution of \eqref{eqqq3.8}, then
 \begin{equation*}
 \aligned
\mc{Z}_1&=\mc{D}^{\dagger}\,\s\,\mc{D}\,\s\,\mc{Z}_1+\mc{Z}_1-\mc{D}^{\dagger}\,\s\,\mc{D}\,\s\,\mc{Z}_1=\mc{D}^{\dagger}\,\s\,\mc{B}+(\mc{I}-\mc{D}^{\dagger}\,\s\,\mc{D})\,\s\,\mc{Z}_1\\ &=\mc{D}^{\dagger}\,\s\,\mc{D}^k\,\s\,\mc{S}+(\mc{I}-\mc{D}^{\dagger}\,\s\,\mc{D})\,\s\,\mc{Z}_1=\mc{D}^{\dagger}\,\s\,\mc{D}^k\,\s\,(\mc{D}^k)^{\dagger}\,\s\,\mc{D}^k\,\s\,\mc{S}+(\mc{I}-\mc{D}^{\dagger}\,\s\,\mc{D})\,\s\,\mc{Z}_1\\
&=\mc{D}^{\dagger}\,\s\,\mc{D}\,\s\,\mc{D}^\D\,\s\,\mc{D}^k\,\s\,(\mc{D}^k)^{\dagger}\,\s\,\mc{B}+(\mc{I}-\mc{D}^{\dagger}\,\s\,\mc{D})\,\s\,\mc{Z}_1\\
&=\mc{D}^{\dagger}\,\s\,\mc{D}\,\s\,\mc{D}^{\ep}\,\s\,\mc{B}+(\mc{I}-\mc{D}^{\dagger}\,\s\,\mc{D})\,\s\,\mc{Z}_1= {\mc{D}^{\dagger,\ep}\,\s\,\mc{B}+(\mc{I}-\mc{D}^{\dagger}\,\s\,\mc{D})\,\s\,\mc{Z}_1.}
\endaligned
 \end{equation*}
 Hence $\mc{Z}_1$ is of the form \eqref{eqqq3.8} and completes the proof.
\end{proof}

\begin{corollary}\label{corr3.8}
    Let 
    $k=\mathrm{ind}(\mc{D})$.
    Then $A^{\dagger,\ep}\,\s\,\mc{B}$ is the unique solution to
\begin{equation}\label{eqq5.9}   \mc{D}\,\s\,\mc{Z}=\mc{D}^k\,\s\,(\mc{D}^k)^{\dagger}\,\s\,\mc{B},\ \ \mc{B}\in \mathbb C^{\textbf{N}(n)},
\end{equation}
in the range space $\rg(\mc{D}^{\dagger}\,\s\,\mc{D}^k)$.
\end{corollary}
\begin{proof}
 According to Theorem \ref{thmm3.6}, it is enough to show the uniqueness of the solution. If $\mc{Z}_1,\mc{Z}_2\in\rg(\mc{D}^{\dagger}\,\s\,\mc{D}^k)$ are two solution of \eqref{eqq4.3}, then
\begin{equation*}
    \mc{Z}_1-\mc{Z}_2\in \nl(\mc{D})\cap\rg(\mc{D}^{\dagger}\,\s\,\mc{D}^k) \subseteq \nl(\mc{D}^{\dagger,\ep}\,\s\,\mc{D})\cap \rg(\mc{D}^{\dagger,\ep}\,\s\,\mc{D}) = \{0\}.
\end{equation*}
Hence $\mc{D}^{\dagger,\ep}\,\s\,\mc{B}$ is the unique solver of \eqref{eqq5.9} in $\rg(\mc{D}^{\dagger}\,\s\,\mc{D}^k)$.
\end{proof}

Using similar lines of the proof of Corollary \ref{corr3.8}, we can show the below results.
\begin{corollary}
 The tensor $A^{\dagger,\ep}\,\s\,\mc{B}$ is the unique solution to the multilinear system
\begin{equation*}
\mc{D}\,\s\,\mc{Z}=\mc{B},~\mc{B}\in \rg(\mc{D}^k),\ \ k=\mathrm{ind}(\mc{D})
\end{equation*}
in the range space $\rg(\mc{D}^{\dagger}\,\s\,\mc{D}^k)$.
\end{corollary}

\begin{corollary}
    Let 
$k=\mathrm{ind}(\mc{D})$. Then $A^{\ep,\dagger}\,\s\,\mc{B}$ is the unique solution to the multilinear system
\begin{equation*}    \mc{D}\,\s\,\mc{Z}=\mc{B},~\mc{B}\in \rg(\mc{D}^k),
\end{equation*}
in the range space $\rg(\mc{D}^{\ep}\,\s\,\mc{D}^k)$.
\end{corollary}

\section{Numerical examples}\label{SecExamples}

Various TGBIs are implemented and tested using the implementation developed in MATLAB, R2022b.
Examples were developed by a personal computer with a CPU [12-Core Intel Xeon W   3.3 GHz], 96GB of memory, and the macOS Ventura operating system.

Now, we restate the Frobenius norm $|| \cdot||_F$  of a tensor $\mc{D}\in\mathbb{C}^{\textbf{M}(m)\times \textbf{N}(n)}$, as proposed in \cite{BrazellT}.
\begin{equation*}
\|\mc{D}\|_F=\left(\sum_{\textbf{M}(m),\textbf{N}(n)}\left|a_{\textbf{M}(m),\textbf{N}(n)}\right|^2\right)^{1/2}.
\end{equation*}
The residual errors associated to different generalized inverses are summarised in below Table \ref{tab:error}.
\begin{table}[htb!]
    \centering
    \caption{Residual errors associate to generalized inverses}
    \vspace*{0.2cm}
    \renewcommand{\arraystretch}{1.1}
    \begin{tabular}{l|l}
    \hline
     $\mc{E}_{\dagger} = \|\mc{D}\,\s\,\mc{D}^{\dagger}\,\s\,\mc{ B} -\mc{ B}\|_F$  & $\mc{E}_{D} = \|\mc{D}\,\s\,\mc{D}^{\D}\,\s\,\mc{ B} -\mc{ B}\|_F$\\
     \hline
      $\mc{E}_{\ep} =\|\mc{D}\,\s\,\mc{D}^{\ep}\,\s\,\mc{ B} -\mc{ B}\|_F$ &
      $\mc{E}_{c, \dagger} =\|\mc{D}\,\s\,\mc{D}^{c,\dagger}\,\s\,\mc{ B} -\mc{ B}\|_F$  \\
       \hline
       $\mc{E}_{\dagger,\D} = \|\mc{D}\,\s\,\mc{D}^{\dagger,\D}\,\s\,\mc{ B} -\mc{ B}\|_F$&$\mc{E}_{D, \dagger} = \|\mc{D}\,\s\,\mc{D}^{\D,\dagger}\,\s\,\mc{ B} -\mc{ B}\|_F$\\
         \hline
       $\mc{E}_{\dagger, \ep} =\|\mc{D}\,\s\,\mc{D}^{\dagger,\ep}\,\s\,\mc{ B} -\mc{ B}\|_F$  & $\mc{E}_{\ep, \dagger} = \|\mc{D}\,\s\,\mc{D}^{\ep,\dagger}\,\s\,\mc{ B} -\mc{ B}\|_F$\\
         \hline
    \end{tabular}
       \label{tab:error}
\end{table}

\begin{example}\rm\cite[Example 5.3]{bilat}\label{pde-1}
Consider the two-dimensional Poisson's equation
\begin{equation*}\label{pdeq-1}
    -\frac{\partial^2 u}{\partial x^2}-\frac{\partial^2 u}{\partial y^2}=f(x,y),\ \ (x,y)\in \Omega=[0,1]\times [0,1]
\end{equation*}
 with $f(x,y)=0$ on the boundary $\partial \Omega$.
An application of the finite difference method for both second order partial derivatives leads to the tensor equation
 \begin{equation*}
     \mc{D}*_2\mc{Z}=\mc{B},\mbox{ where } ~\mc{Z}\in\mathbb{R}^{n\times n},~\mc{B}\in\mathbb{R}^{n\times n}, \mbox{ and }\mc{D}=\mathrm{tridiagonal}(\mc{P},\mc{Q},\mc{P})\in\mathbb{R}^{n\times n\times n\times n}
 \end{equation*}
with
\begin{equation*}
    \mc{P}=\begin{pmatrix}
-2 & -1 & & 0\\
-1 & \ddots & \ddots & \\
& \ddots & \ddots & -1 \\
0 & & -1 & -2 \end{pmatrix},~~\mc{Q}=\begin{pmatrix}
24 & -4 & & 0\\
-4 & \ddots & \ddots & \\
& \ddots & \ddots & -4 \\
0 & & -4 & 24 \end{pmatrix}.
\end{equation*}
Here the tensor $\mc{D}$ is non-singular, and we create three singular tensors $\mc{D}_i=\mc{D}\bigoplus \mc{N}_i$ with different indices, where \\
$\mc{N}_1=\begin{pmatrix}
    0 & 1 & 0\\
    0 & 0 & 1\\
    0 &0 &0
\end{pmatrix},~\mc{N}_2=\begin{pmatrix}
    0 & 1 & 0 & 0\\
    0 & 0 & 1 & 0\\
    0 &0 &0 & 1\\
    0 & 0 & 0& 0
\end{pmatrix}$ and $\mc{N}_3=\begin{pmatrix}
    0 & 1 & 0 & 0& 0\\
    0 & 0 & 1 &0 & 0\\
    0 &0 &0 & 1 & 0\\
    0& 0 & 0 & 0& 1\\
    0 & 0 &0 & 0& 0
\end{pmatrix}$. \\
The notation nnz($\mc{D}$) is used to compute the total number of nonzero elements in the tensor $\mc{D}$.
It is easy to see that $\mathrm{ind}(\mc{D}_1)=3$, $\mathrm{ind}(\mc{D}_2)=4$ and $\mathrm{ind}(\mc{D}_3)=5$.
For a fixed $\mc{B} \in \rg(\mc{D}^{k})$, and running the {\em MATLAB} code 30 times, we compute the average residual errors and mean CPU times (in seconds), which are presented in Table \ref{tab:err-comp}.

\begin{table}[!htb]
    \centering
      \caption{Comparison analysis of residual error and mean CPU time}
      \vspace*{0.2cm}
       \renewcommand{\arraystretch}{1.1}
   \begin{tabular}{|c |c|c|c|c|c|c }
\hline
$i$ &Order of $\mc{D}_i$& $\mathrm{ind}(\mc{D}_i)$ & nnz($\mc{D}_i$) & Residual error &Mean CPU time \\
\hline
\multirow{8}{*}{1}&\multirow{8}{*}{$19\times 337\times 19\times 337$} & \multirow{8}{*}{3}&  \multirow{8}{*}{56646} &$\mc{E}_{\dagger}= 9.6371e^{-10}$  & 0.627935 \\ \cline{5-6}
& &  & &$\mc{E}_{D}=4.0672e^{-09}$ & 0.633671\\ \cline{5-6}
& &  & &  $\mc{E}_{\ep}=9.7280e^{-09}$& 0.625599\\ \cline{5-6}
& &  & &$\mc{E}_{c,\dagger}=4.5532e^{-09}$ & 0.624649  \\ \cline{5-6}
& &  & &  $\mc{E}_{\dagger,D}= 4.2339e^{-09}$ & 0.633521 \\ \cline{5-6}
& & & & $\mc{E}_{D,\dagger}=4.2277e^{-09}$ & 0.644424 \\\cline{5-6}
& &  & & $\mc{E}_{\dagger,\ep}=9.7934e^{-09}$& 0.638693  \\\cline{5-6}
& &  & &$\mc{E}_{\ep,\dagger}=9.7907e^{-09}$ & 0.651126  \\\cline{5-6}
\hline
\multirow{8}{*}{2}&\multirow{8}{*}{$74\times 100\times 74\times 100$} & \multirow{8}{*}{4}&  \multirow{8}{*}{65539} &$\mc{E}_{\dagger}=2.1644e^{-08}$  & 0.833736 \\ \cline{5-6}
& &  & &$\mc{E}_{D}=9.7385e^{-08}$ &0.867778\\ \cline{5-6}
& &  & &  $\mc{E}_{\ep}=3.1213e^{-07}$& 0.879092 \\ \cline{5-6}
& &  & &$\mc{E}_{c,\dagger}=1.0686e^{-07}$ & 0.869206   \\ \cline{5-6}
& &  & &  $\mc{E}_{\dagger,D}=1.0012e^{-07}$ & 0.851221  \\ \cline{5-6}
& & & & $\mc{E}_{D,\dagger}=1.0006e^{-07}$ & 0.858702 \\\cline{5-6}
& &  & & $\mc{E}_{\dagger,\ep}=3.1247e^{-07}$&  0.869249 \\\cline{5-6}
& &  & &$\mc{E}_{\ep,\dagger}=3.1245e^{-07}$ & 0.853883 \\\cline{5-6}
\hline
\multirow{8}{*}{3}&\multirow{8}{*}{$87\times 115\times 87\times 115$} & \multirow{8}{*}{5}&
\multirow{8}{*}{88808} &$\mc{E}_{\dagger}=6.4945 e^{-07}$  & 2.243536 \\ \cline{5-6}
& &  & &$\mc{E}_{D}= 4.5643e^{-07}$& 2.228076 \\ \cline{5-6}
& &  & &  $\mc{E}_{\ep}= 1.2041e^{-07}$& 2.233148 \\ \cline{5-6}
& &  & &$\mc{E}_{c,\dagger}= 4.7523e^{-07}$ &   2.224729\\ \cline{5-6}
& &  & &  $\mc{E}_{\dagger,D}=4.6161e^{-07}$ & 2.228884  \\ \cline{5-6}
& & & & $\mc{E}_{D,\dagger}= 4.6155e^{-07}$ &  2.252194 \\\cline{5-6}
& &  & & $\mc{E}_{\dagger,\ep}=1.2071e^{-07}$&  2.242604 \\\cline{5-6}
& &  & &$\mc{E}_{\ep,\dagger}=1.2071e^{-07}$ & 2.245622 \\\cline{5-6}
\hline
\end{tabular}
      \label{tab:err-comp}
\end{table}
\end{example}

\begin{example}\rm\cite[Example 6]{AsishRJ}\label{sol-pde}
    Consider the two-dimensional Poisson's equation
\begin{equation}\label{PDE11}
    -\frac{\partial^2 u}{\partial x^2}-\frac{\partial^2 u}{\partial y^2}=f(x,y),~(x,y)\in \Omega=[0,1]\times [0,1]
\end{equation}
 with $\frac{\partial u}{\partial {\bf n}}=0$ on the boundary $\partial\Omega$, where ${\bf n}$ stands for the normal to the boundary $\partial\Omega$.
 The finite difference method for both second order partial derivatives lead to the following tensor equation
 \begin{equation}\label{multi}
     \mc{D}*_2\mc{Z}=\mc{B}, \ \ \mc{Z}\in\mathbb{R}^{n\times n},~\mc{B}\in\mathbb{R}^{n\times n}.
 \end{equation}
The coefficient tensor $\mc{D}\in\mathbb{R}^{n\times n\times n\times n}$ is of the form
\begin{equation*}\label{F-ex-ten}
\mc{D}=\mc{I}_n\kronecker \mc{B} +\mc{C}\kronecker \mc{I}_n+\mc{D},
\end{equation*}
where $\mc{I}_n\in\mathbb{R}^{n\times n}$ is the identity tensor,
\begin{equation*}
    \mc{B}=\mathrm{tridiagonal}\left(-1,0,-1\right) = \begin{pmatrix}
0 & -1 & & 0\\
-1 & \ddots & \ddots & \\
& \ddots & \ddots & -1 \\
0 & & -1 & 0 \end{pmatrix}=
 \mc{C}\in \mathbb{R}^{n\times n},
\end{equation*}
and $\mc{D}\in\mathbb{R}^{n\times n\times n\times n}$ is a diagonal tensor in which diagonal entries change according to the number of grid points chosen.
It can be easily verified that the tensor $\mc{D}$ is singular and $\mathrm{ind}(\mc{D}) = 1$.
Therefore, $\mc{D}^{\D}=\mc{D}^{\#}$, $\mc{D}^{\ep}=\mc{D}^{\core}$, $\mc{D}^{\dagger}=\mc{D}^{c,\dagger}=\mc{D}^{\dagger,\core}$, $\mc{D}^{\core}=\mc{D}^{\#,\dagger}=\mc{D}^{\core,\dagger},~\mc{D}^{\dagger,\#}=\mc{D}^{\dagger}\,\s\,\mc{D}\,\s\,\mc{D}^{\#}$.

\begin{figure}[!htb]
\begin{center}
\begin{tabular}{cc}
{\hspace*{-0.5cm}\scalebox{0.55}{\includegraphics{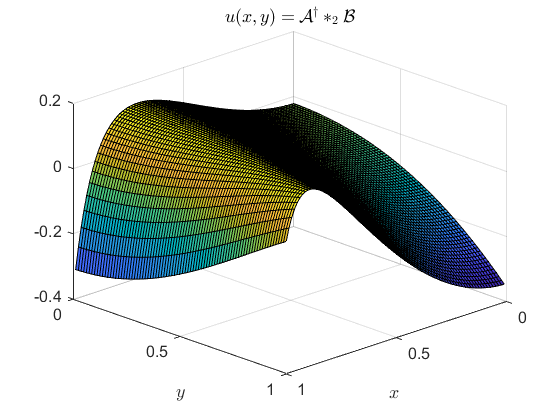}}}&
 {\hspace*{-0.5cm}\scalebox{0.55}{\includegraphics{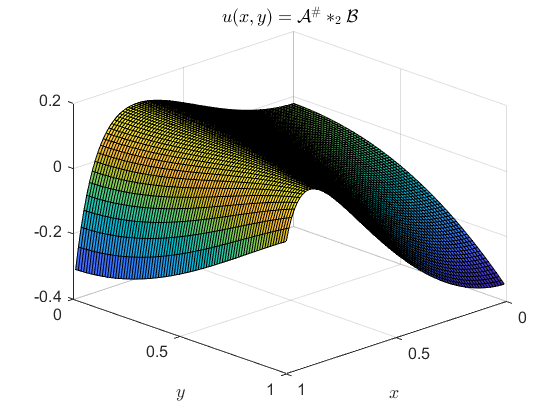}}}\\
  {\hspace*{-0.5cm}\scalebox{0.55}{\includegraphics{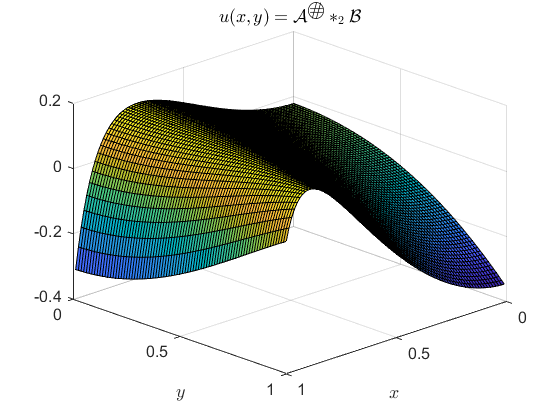}}}&
 {\hspace*{-0.5cm}\scalebox{0.55}{\includegraphics{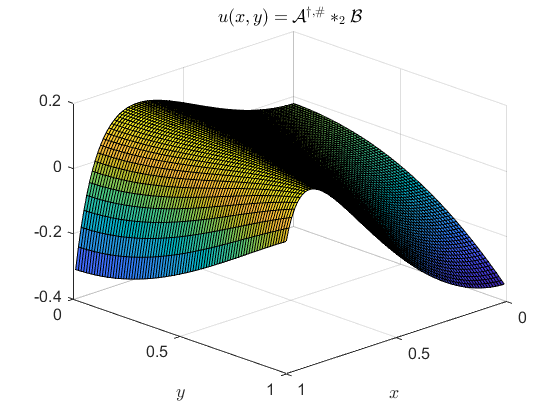}}}
\end{tabular}
\vskip -6pt 
\caption{Solutions of the multilinear system \eqref{multi} for different generalized inverses when $n=80$.}
\label{fig-1}
\end{center}
\end{figure}

\begin{figure}[!htb]
\begin{center}
\begin{tabular}{cc}
{\hspace*{-0.5cm}\scalebox{0.5}{\includegraphics{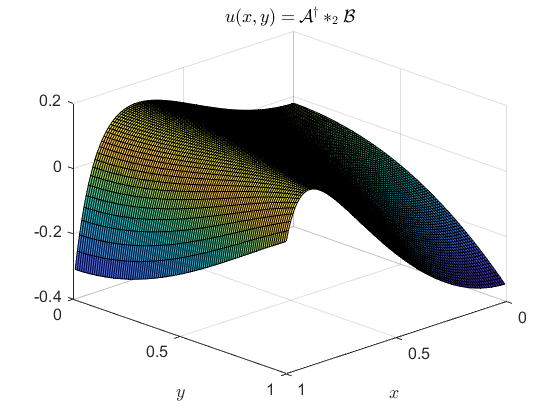}}}&
 {\hspace*{-0.5cm}\scalebox{0.5}{\includegraphics{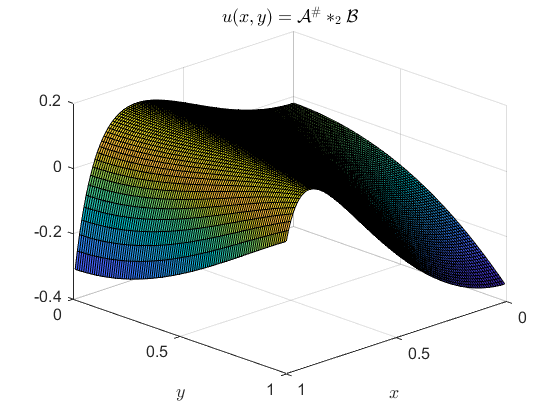}}}\\
  {\hspace*{-0.5cm}\scalebox{0.5}{\includegraphics{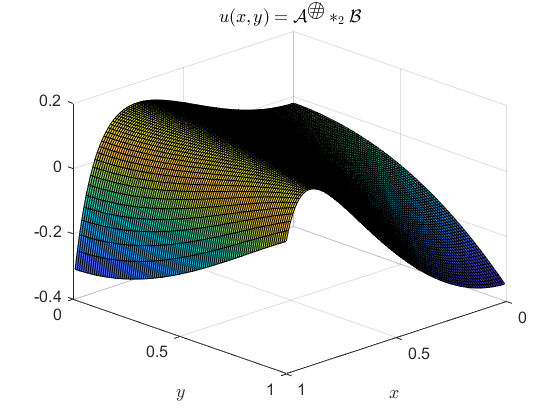}}}&
 {\hspace*{-0.5cm}\scalebox{0.5}{\includegraphics{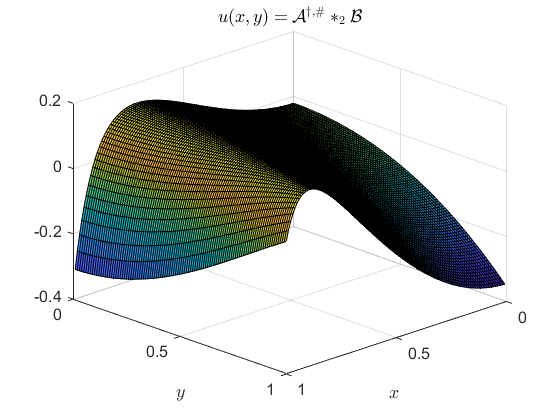}}}
\end{tabular}
\vskip -6pt 
\caption{Solutions of the multilinear system \eqref{multi} for different generalized inverses when $n=100$.}
\label{fig-2}
\end{center}
\end{figure}
\end{example}
The notation $R^{T}_{errors}$ is used for residual error, and mean CPU time (in seconds) is denoted by  Mean Time$^{T}$, where the computations are done in the tensor framework with respect to the Einstein product as defined in \eqref{Eins}. Tensor computation residual errors and mean CPU time for different values of $n$, are given in Table \ref{tab:mat-ten-comp}.

\begin{table}[!htb]
    \centering
    \caption{Tensor computation residual errors and CPU time}
    \vspace*{0.1cm}
    \begin{tabular}{|c|c|c|c|c}
    \hline
     Order of $\mc{D}$  & $R^{T}_{errors}$ & Mean Time$^{T}$ \\
         \hline 
       \multirow{4}{*}{$80\times 80\times 80\times 80$}& $E_{\dagger}=1.9751e^{-12}$ &  1.617058  \\

          &$E_{\#}=1.9635e^{-11}$ &  1.624961 \\

          & $E_{\core}=1.7284e^{-11}$&   1.676089\\

          & $E_{\dagger,\#}=1.7186e^{-11}$&  1.670214 \\
         \hline
          \multirow{4}{*}{$100\times 100\times 100\times 100$}& $E_{\dagger}= 3.0966e^{-12}$ &2.18825  \\

          &$E_{\#}=1.3118e^{-11}$ &2.205259\\

          & $E_{\core}=1.7072e^{-11}$& 2.278892 \\

          & $E_{\dagger,\#}= 1.69869e^{-11}$ & 2.341240  \\
         \hline
            \multirow{4}{*}{$120\times 120\times 120\times 120$}& $E_{\dagger}=4.3320e^{-12}$ & 6.167582 \\

          &$E_{\#}=3.0744e^{-11}$ &  6.141843 \\

          & $E_{\core}=3.4814e^{-12}$& 6.320301  \\

          & $E_{\dagger,\#}=3.4353e^{-12}$&6.049212  \\
         \hline
    \end{tabular}
    \label{tab:mat-ten-comp}
\end{table}

Additionally, we discuss the approximate solution the Poisson's equation \eqref{PDE11} with respect to different choices of generalized inverses and order $n$ tensors.
By considering a fixed $\mc{B}\in \rg(\mc{D})$, figures \ref{fig-1} and \ref{fig-2} represent plots of the solutions ($\mc{D}^{\dagger}*_2\mc{B},~\mc{D}^{\#}*_2\mc{B},~\mc{D}^{\core}*_2\mc{B},~\mc{D}^{\dagger,\#}*_2\mc{B}$) to the multilinear system $\mc{D}*_2\mc{X}=\mc{B}$.
Plots of solutions to the multilinear system \eqref{multi} based on different generalized inverses in the case $n=80$ are Figure \ref{fig-1}.
Plots of solutions to the multilinear system \eqref{multi} based on different generalized inverses in the case $n=100$ are Figure \ref{fig-2}.
Figure \ref{fig-1} and Figure \ref{fig-2} show that the use of different generalized inverses does not affect the final results.

\section{Conclusion}\label{SecConclusion}

We introduce  {the} class of tensor generalized bilateral inverses (TGBIs) under the Einstein product. The TGBIs class is an extension of the class of generalized bilateral inverses (GBIs) in a matrix environment. Moreover, the TBGI class includes so far considered composite generalized inverses (CGIs) for matrices and tensors. Applications of TBGIs in solving certain multilinear systems have been presented.
{
 Some characterizations and representations of TGBI, along with numerical examples, were studied. Further, we have developed a few characterizations of known CGIs, such as CMP, DMP, MPD, MPCEP, and CEPMP.} One of the possible topics for further research is investigating the perturbation theory on the bilateral generalized inverses.
{Weighted tensor generalized bilateral inverses may be another interesting area for future research.}

\medskip

{ {
\noindent{\bf Funding}
\begin{itemize}
\item Ratikanta Behera is supported by the Science and Engineering Research Board (SERB), Department of Science and Technology, India, under Grant No. EEQ/2022/001065.
\item Jajati Keshari Sahoo is supported by the Science and Engineering Research Board (SERB), Department of Science and Technology, India, under Grant No. SUR/2022/004357.
\item Predrag Stanimirovi\' c is supported by Science Fund of the Republic of Serbia (No. 7750185, Quantitative Automata Models: Fundamental Problems and Applications - QUAM). \\
Further, Predrag Stanimirovi\' c is supported by the Ministry of Science and Higher Education of the Russian Federation (Grant No. 075-15-2022-1121).
\end{itemize}
}

\smallskip
\noindent {\bf Declaration of competing interests.}
The authors declare that they have no known competing financial interests or personal relationships that could have appeared to influence the work reported in this paper.

\bibliographystyle{abbrv}
\bibliography{reference}
\end{document}